\documentclass[sn-mathphys,Numbered]{sn-jnl}

\usepackage{graphicx}%
\usepackage{multirow}%
\usepackage{amsmath,amssymb,amsfonts}%
\usepackage{amsthm}%
\usepackage{mathrsfs}%
\usepackage[title]{appendix}%
\usepackage{xcolor}%
\usepackage{textcomp}%
\usepackage{manyfoot}%
\usepackage{booktabs}%
\usepackage{algorithm}%
\usepackage{algorithmicx}%
\usepackage{algpseudocode}%
\usepackage{listings}%
\usepackage{caption}
\usepackage{float}
\usepackage{subcaption}
\usepackage[shortlabels]{enumitem}
\usepackage{enumitem} 
\usepackage{mathrsfs}
\usepackage{multirow}
\usepackage{subcaption}
\usepackage{hyperref}
\usepackage[utf8]{inputenc}
\usepackage{xspace}

\theoremstyle{thmstyleone}%
\newtheorem{theorem}{Theorem}
\newtheorem{problem}[theorem]{Problem}%
\theoremstyle{thmstyletwo}%
\newtheorem{example}{Example}%
\newtheorem{remark}{Remark}%

\theoremstyle{thmstylethree}%
\newtheorem{definition}{Definition}%

\begin{document}

\title[Article Title]{Steady State Classification of Allee Effect System}


\author[]{\fnm{Kuo} \sur{Song}}

\author*[]{\fnm{Xiaoxian} \sur{Tang}}\email{xiaoxian@buaa.edu.cn}

\affil*{\orgdiv{School of Mathematical Sciences}, \orgname{Beihang University}, \orgaddress{\city{Beijing}, \postcode{102200}, \country{China}}}




\abstract{In this paper, we consider the steady state classification problem of the Allee effect system for multiple tribes. First, we reduce the high-dimensional model into several two-dimensional and three-dimensional algebraic systems such that we can prove a comprehensive formula of the border polynomial for arbitrary dimension. Then, we propose an efficient algorithm for classifying the generic parameters according to the number of steady states, and we successfully complete the computation for up to the seven-dimensional Allee effect system.
 }




\maketitle

\section{Introduction}\label{section:Introduction}
Solving steady states of the Allen effect system is usually a key problem in many  applied areas. For instance, in evolutionary biology,
 the close correlation  between Allee effects and the risk of population extinction has been discussed since over fifteen  years ago in \cite{bib_bib1}, and  particularly in the context of protecting endangered species and ecosystems,  the multiple Allee effects has great impact  on population management strategies.  Besides,  the Allee effect can be interesting for its impact on profit margins in economics \cite{bib_bib2}.
Recently, in the context of epidemics, the Allee effect is used to study the relationship between the vaccination and the threshold for herd immunity \cite{bib_bib3}. And in artificial intelligence, the Allee effect helps to explore the community formation and the network stability \cite{bib_bib4}. 
Since Allee effect  is so important, in this work, we are interested in how many steady states a population model might admit when it incorporates the Allee effect, which is a cutting edge problem in algebraic biology.  Recall that Gergely Röst and AmirHosein Sadeghimanesh have presented the Allee effect model  \cite[Equation (3)]{bib_bib5}, and we adopt their description as follows,
\begin{equation} 
\dot{x}_i = x_i(1 - x_i)(x_i - b) - (n - 1)ax_i + \sum_{\substack{j=1 \\ j \neq i}}^n a x_j, \quad i = 1, \cdots, n.\label{eq:case1}
\end{equation}
where $x_i\in {\mathbb R}_{\ge 0}$ denotes the population size of the patch, the parameter $a\in {\mathbb R}_{\ge 0}$ denotes the spatial dispersal rate, and the parameter $b\in [0,0.5]$ denotes the Allee threshold.  
We remark that the parameter $b$ is originally nonnegative, and it is known that it is sufficient to consider its value over the interval $[0,0.5]$ due to some symmetry of the system \eqref{eq:case1}, see more details in \cite[Lemmas 2.1 and 2.2]{bib_bib5}.
 Here, by ``steady state classification" we mean to classify the parameters (the spatial dispersal rate and the Allee threshold) according to the number of nonnegative steady states (the population sizes). 

Such a problem can be easily formulated as a real quantifier elimination problem. It is well known
that the real quantifier elimination problem can be carried out by the famous  cylindrical algebraic decomposition (CAD) method \cite{bib_bib6,bib_bib7,bib_bib8,bib_bib9,bib_bib10,bib_bib11,bib_bib12,bib_bib13,bib_bib14_PhDthesis,bib_bib16,bib_bib17,bib_bib18,bib_bib19,bib_bib20,bib_bib23,bib_bib24,bib_bib25,bib_bib26,bib_bib27,bib_bib28,bib_bib29,bib_bib30,bib_bib21,bib_bib22,bib_bib15}. There are several software systems such as QEPCAD \cite{bib_bib15,bib129,bib127,bib_bib22}, Redlog \cite{bib129}, Reduce (in Mathematica) \cite{bib_bib13,bib131} and SyNRAC \cite{bib132}.
Hence, in principle, the steady state classification of the system \eqref{eq:case1} can be carried out automatically using those software systems. However, it is also well known that the complexity (roughly speaking, double exponential in $n$ \cite{bib133,bib_bib19})  of those algorithms is way beyond current computing capabilities when the dimension $n$ (also, it is the number of variables) is arbitrarily large since those algorithms are for general quantifier elimination problems.

The steady state classification problem is basically a real root
classification problem for semi-algebraic systems, which is a special type of quantifier elimination problem. Hence, it would be advisable to apply the method of real root classification (RRC) \cite{bib134,bib135}. 
The main idea of RRC method is to first deal with the algebraic equations in the semi-algebraic systems. 
In the standard methods, there are two ways for doing this: (i) computing a triangular decomposition \cite{bib_Wu1958}, or (ii) computing a Gr\"obner basis \cite{bib_Gronber}.  Depending on the two approaches for solving the equations, there are two concepts ``border polynomial (BP)" \cite{bib134} and ``discriminant variety (DV)" \cite{bib63}, which can be considered as the generalizations of the discriminant for a univariate polynomial. Both methods are more practical than a standard CAD for a general zero-dimensional system.  In \cite{bib_bib5}, Gergely Röst and AmirHosein Sadeghimanesh have classified the steady states of the system \eqref{eq:case1} for $n=2,3$ by applying CAD to the DV. However, in general, the real root classification method might not go beyond these, due to enormous computing time/memory requirements for large $n$ (e.g., one can take a look at the computational timings recorded in  Table \ref{table_1}).

In this work, we propose a novel method for efficiently computing the border polynomials and classifying the steady states of the system \eqref{eq:case1} for any $n$. Briefly, we have the following contributions. 

\begin{itemize}
\item[(I)] For any $n\in {\mathbb N}_{+}$, we have proved a comprehensive formula for a border polynomial of the system \ref{eq:case1}, see Theorem \ref{thm:discriminant}. Experiments show that one can easily compute a border polynomial by Theorem \ref{thm:discriminant} for large $n$ (for instance, it takes only half a minute for $n=100$) while the standard tools can not finish the computation in an acceptable time even for $n=7$, see Table \ref{table_1}. 
\item[(II)] We provide a novel algorithm for classifying the steady states of the system \ref{eq:case1}, see Section \ref{section:Solving Scheme-Region solution}. We have successfully classified all the generic parameters according to the number of steady states for $n=4, 5, 6, 7$. 
\end{itemize}
In the proof of the main result (Theorem \ref{thm:discriminant}), the crucial challenge we have to tackle is to derive a border polynomial when the dimension is 
arbitrarily high. Here, we overcome this problem by studying the structure of the system \eqref{eq:case1}. In fact, we find that any high dimensional system can be reduced to several two-dimensional and three-dimensional systems (Theorem \ref{theorem 1}). The idea is inspired by the second author's old work \cite{bib_bib}, where a high dimensional regulated biological system can be cut down as finitely many two-dimensional systems (but here, we also need to deal with three-dimensional systems). Then, we derive comprehensive formulas for the border polynomials of these low-dimensional systems by running {\tt Maple} command {\tt BorderPolynomial}, which was originally developed as a main function in an algebraic software called {\tt DISCOVERER} \cite{bib134}, and was combined with other tools for solving parametric semi-algebraic systems into the {\tt Maple} package {\tt RegularChains} later \cite{bib_49}.

The rest of this paper is organized as follows. In Section \ref{section:Problem Statement}, we recall  the basic definitions for the steady state, and we formally state the steady state classification problem for the system \eqref{eq:case1}. In Section \ref{section:Solving Scheme}, we present an algorithm for solving the problem. More specifically, in Section \ref{section:Theorem Proof}, we prove that the coordinates of any steady state of the system \eqref{eq:case1} consist of at most three distinct positive numbers (Theorem \ref{theorem 1}). In Section \ref{section:Solving Scheme-Discriminant},  we recall the definition of border polynomial and we derive a comprehensive formula for the border polynomials (Theorem \ref{thm:discriminant}). Also, we compare  our method with two standard methods for computing discriminants (Table \ref{table_1}).  In Section \ref{section:Solving Scheme-Region solution}, we present an algorithm for classifying the steady states. In Section \ref{section:Examples}, we illustrate how the algorithm works by an example for $n = 4$, and we present the classification results for $n=4,5,6,7$. Finally, we end this paper with some future directions inspired by the results, see Section \ref{section:Conclusion and Discussion}.
\section{Problem Statement}\label{section:Problem Statement}

We call $x^*\in {\mathbb R}^n$ is a \textcolor{blue}{\rm steady state} of the system \eqref{eq:case1} for any given pair of parameters $(a, b)\in {\mathbb R}^2$, if the right-hand side of the system vanishes at the point $x^*$.  Notice that only nonnegative steady states (i.e., $x^*\in {\mathbb R}_{\geq 0}^n$) are biologically meaningful. So, in the rest of this paper, when we say ``steady states", we mean ``nonnegative steady states".  

\begin{problem}\label{Problem}
    Given $n\in {\mathbb N}_{+}$, we need to classify the generic parameters $(a, b)\in {\mathbb R}_{\ge 0}\times [0,0.5]$ according to  the number of the steady states of the system \eqref{eq:case1}. That is to say, we hope to efficiently determine the population density of each tribe at the steady state  for generic Allee threshold and spatial dispersal rate between patches.\\
\end{problem}
Note that for $n=1$, the ODE system \eqref{eq:case1} becomes
           \begin{align*}
           \dot{x}_1 =x_1(1 - x_1)(x_1 - b).
               \end{align*}
          Easily, we see that there are three steady states $0$, $1$ and $b$. Also, recall that for $n=2,3$, Problem \ref{Problem} has been solved in \cite{bib14,bib_bib5}.
          So, in this work, we will focus on  Problem \ref{Problem} for $n\geq 4$.
\section{Steady State Classification}\label{section:Solving Scheme}
In this section, we discuss how to efficiently  solve Problem \ref{Problem} for the system \eqref{eq:case1}. First, we prove a theorem (Theorem \ref{theorem 1}) in section \ref{section:Theorem Proof}, which says that for any $n\in {\mathbb N}_{+}$, the system \eqref{eq:case1} can be reduced  into several subsystems with dimension two or three.
After that we derive a comprehensive formula for the border polynomial (Theorem \ref{thm:discriminant}) in section \ref{section:Solving Scheme-Discriminant}. Then, we can develop an efficient algorithm for classifying the steady states in section \ref{section:Solving Scheme-Region solution}. At last, we present the computational results  in section \ref{section:Examples}.
\subsection{Dimension Reduction Theorem}\label{section:Theorem Proof}
\begin{theorem}\label{theorem 1}
 Consider the system of ODEs in \eqref{eq:case1}.
 For any steady state $x=(x_1,\cdots,x_n)\in {\mathbb R}_{\ge 0}^n$, the coordinates of $x$ consist of at most three distinct positive numbers.
\end{theorem}
\begin{proof}
For any $z\in {\mathbb R}$, we define a univariate function 
\begin{align*}
 g(z) &:= -z(1-z)(z-b) + naz,
\end{align*}
where $n \in {\mathbb{N}_{+}}$ denotes the dimension, and $a$ and $b$ are real parameters described as in the system \eqref{eq:case1}.
\begin{figure}[H]
    \centering
\includegraphics[width=0.4\linewidth]{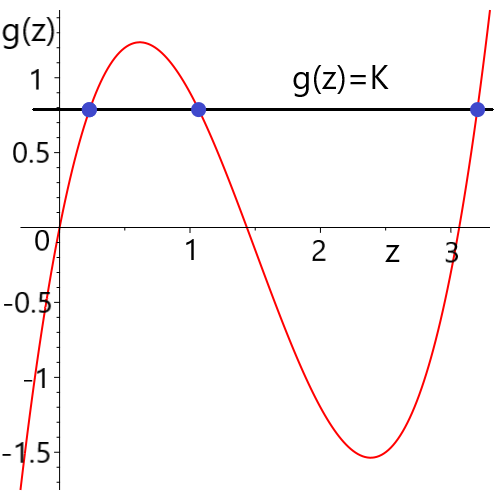}
    \caption{$g(z)=K$ has at most three real solutions.}
    \label{Fig.proof1}
\end{figure}
\noindent
For any  $x=(x_1, \ldots, x_n)\in {\mathbb R}^n$, we define a multivariate function 
\begin{align*}
 s(x) &:= a\sum_{i=1}^nx_i.
 \end{align*}
Then, any steady state of 
the system \eqref{eq:case1} is a common real root of the following polynomials:
\begin{align}
 f_i &:= -g(x_i)+s(x),\quad \quad i=1,\cdots, n.\label{eq:sysf}
\end{align}
Note that $g(z)$ is a cubic polynomial in ${\mathbb R}[z]$. So, for any $K\in {\mathbb R}$, $g(z)=K$ has at most three real solutions, see Fig. \ref{Fig.proof1}. 
Therefore, 
 for any steady state $x=(x_1,\cdots,x_n)\in {\mathbb R}_{\ge 0}^n$ of 
 the system \eqref{eq:case1}, the coordinates of $x$ consist of at most three distinct positive numbers.
\end{proof}
\subsection{Computing Border Polynomials}\label{section:Solving Scheme-Discriminant}
Consider the steady-state system of the ODE system \eqref{eq:case1}:
\begin{align}\label{eq:SAS}
f_i(a,b,x_1,\ldots,x_n)=0, \;& i=1, \ldots, n\notag\\
x_i\geq 0, \; & i=1, \ldots, n \notag\\
a\geq 0, \;& 0\leq b\leq \frac{1}{2} 
\end{align}
where $f_i$ is defined as in \eqref{eq:sysf}.
Below, we present the definition of border polynomial for the system \eqref{eq:SAS}. See the definition for a more general semi-algebraic set in \cite{bib135}. 
\begin{definition}\label{def:bp}\cite[Definition 6.1]{bib135}
Consider the semi-algebraic system \eqref{eq:SAS} in ${\mathbb Q}[a, b, x]$.  If a polynomial $q(a, b)\in {\mathbb Q}[a, b]$ satisfies
\begin{itemize}
    \item[(a)] the system \eqref{eq:SAS}  has only finitely many real solutions for any $(a, b)\in \mathbb{R}^2$ satisfying $q(a, b) \neq 0$, and
    \item[(b)] the number of distinct real solutions of the system \eqref{eq:SAS} is constant in each connected component of $\{(a, b)\in \mathbb{R}^2|q(a, b) \neq 0\}$,
\end{itemize}
then $q(a, b)$ is called a \textcolor{blue}{{\rm border polynomial}} of the system \eqref{eq:SAS}.
\end{definition}

The goal of  this section is to derive a border polynomial of the system \eqref{eq:SAS}. The main ideal is to first reduce the system \eqref{eq:SAS} into several subsystems with small dimensions  according to Theorem \ref{theorem 1}. In fact, by Theorem \ref{theorem 1},
for any steady state $x=(x_1,\cdots,x_n)\in {\mathbb R}_{\ge 0}^n$  of the system \eqref{eq:case1}, 
 we need to deal with the following cases. 
\begin{enumerate}[label={\textbf{Case \arabic*}},ref=Case \arabic*, start=0]
    \item\label{Case_0} 
           If $x_1=\cdots=x_n=y$, then the steady-state system \eqref{eq:case1} becomes one equation 
          \begin{equation}
           \begin{aligned}
           y(1 - y)(y - b)=0.\label{eq:case2}
               \end{aligned}
          \end{equation}
          Easily, we get three solutions $y=0,\;y=1,\;y=b$.
          So, the system \eqref{eq:case1} always has three trivial 
          steady states $(0,\cdots, 0)$, $(1,\cdots, 1)$, and $(b,\cdots, b)$ for any parameters.
           \item\label{Case_1} 
    Assume that the coordinates of $x$ consist of two positive numbers $y$ and $z$. Suppose $y$ and $z$ appear in $x$ respectively $n_1$ and $n_2$ times ($n_1+n_2=n,\;n\ge 2$).  Without loss of generality, we assume that $x_1=\cdots=x_{n_1}=y$ and $x_{n_1+1}=\cdots=x_{n}=z$.
    \item\label{Case_2} 
    Assume that the coordinates of $x$ consist of three positive numbers $y$, $z$ and $w$. Suppose $y$, $z$ and $w$ appear in $x$ respectively $n_1$, $n_2$ and $n_3$ times ($n_1+n_2+n_3=n,\;n\ge 3$). Without loss of generality, we assume that  $x_1=\cdots=x_{n_1}=y,\; x_{n_1+1}=\cdots=x_{n_1+n_2}=z$ and $x_{n_1+n_2+1}=\cdots=x_{n}=w$.
\end{enumerate}
In subsections \ref{section:Solution consists of two numbers} and \ref{section:Solution consists of three numbers}, we will respectively reduce the system \eqref{eq:case1}  for \ref{Case_1} and \ref{Case_2}. Then, in subsection \ref{section:Conclusion}, we will derive a comprehensive formula of 
a border polynomial for any dimension $n$. 
\subsubsection{Case 1: Steady State Consisting of Two Numbers}\label{section:Solution consists of two numbers}
Suppose $x$ is a steady state of the system \eqref{eq:case1}. Note that $x$ is also a real solution of the system \eqref{eq:SAS}. 
According to the hypothesis of \ref{Case_1}, we assume that the coordinates of $x$ consist of two positive numbers $y$ and $z$, and they appear in $x$ respectively $n_1$ and $n_2$ times ($n_1+n_2=n,\;n\ge 2$). We substitute $x_1=\cdots=x_{n_1}=y,\;x_{n_1+1}=\cdots=x_n=z$ into the system \eqref{eq:SAS}, and the first $n$ algebraic equations in \eqref{eq:SAS} becomes the following two equations 
          \begin{align}
           \mathscr{G}_{11}(n_1,n_2) &:= y(1 - y)(y - b) - nay + a(n_1y +n_2z) = 0,\label{eq:f11} \\
           \mathscr{G}_{12}(n_1,n_2) &:= z(1 - z)(z - b) - naz + a(n_1y +n_2z) = 0,\label{eq:f12}
           \end{align}
           where
           \begin{align}\label{eq:con12}
           y,z,a\in {\mathbb R}_{\geq 0},\;\; b\in [0,0.5].
           \end{align}
           We call the above system ${\tt \mathscr{G}}_1(n_1, n_2)$. Using ${\tt Maple}$, we compute a border polynomial of the system ${\tt \mathscr{G}}_1(n_1, n_2)$ with the constraints \eqref{eq:con12}, and we obtain
           \begin{align}
            &{\tt bp}_1(a,b;n_1,n_2):=abn_1n_2(b-\frac{1}{2})(729a^9n_1^5n_2^4 + 16a^6n_1^4n_2^2+ \cdots - \frac{1}{64}b^6) \notag\\&(an_1 + an_2 + b)(an_1 + an_2 - b + 1)(an_1 + an_2 + b^2 - b)\label{bp_1(a,b;n_1,n_2)}
          \end{align}
          Here in the above polynomial, we omit $154$ terms. We provide a ${\tt Maple}$ file \footnote{see: https://github.com/songkuo-ux/Allee-Effect/blob/master/3.2.1.mw} containing the full expression of ${\tt bp}_1(a,b;n_1,n_2)$ for the readers to check the computation presented in this section.
Note that if $n_1=n_2$, then we obtain
          \begin{align}
            &{\tt bp}_1(a,b;n_1,n_1)=abn_1(-\frac{1}{2}+ b)(an_1 + \frac{b}{2})(an_1 - \frac{b}{2} + \frac{1}{2})(27a^3n_1^3 - 9a^2b^2n_1^2 \notag \\&+ ab^4n_1 + 9a^2bn_1^2 - 2ab^3n_1 - 9a^2n_1^2 + 3ab^2n_1 - \frac{1}{4}b^4 - 2abn_1 + \frac{1}{2}b^3 + an_1 - \frac{1}{4}b^2)\notag \\&(2an_1 + b^2 - b)\label{bp_1(a,b;n_1,n_1)}
             \end{align}
For instance, for $n_1=3,\;n_2=1$ the polynomial \eqref{bp_1(a,b;n_1,n_2)} becomes \eqref{bp_1(a,b;3,1)}, and for $n_1=n_2=2$ the polynomial \eqref{bp_1(a,b;n_1,n_1)} becomes \eqref{bp_1(a,b;2,2)},
\begin{align}
    &{\tt bp}_1(a,b;3,1) =3ab(b-\frac{1}{2})(\frac{17}{4}ab^6 - \frac{3}{2}ab^5+ \cdots -\frac{1}{64}b^6)(4a + b)(4a - b + 1)(b^2 + 4a\notag\\& - b), \label{bp_1(a,b;3,1)} \\
    &{\tt bp}_1(a,b;2,2) =2ab(b-\frac{1}{2})(2a + \frac{b}{2})(2a -\frac{b}{2} + \frac{1}{2})(216a^3 - 36a^2b^2 + 2ab^4 + 36a^2b \notag\\&- 4ab^3 - 36a^2 + 6ab^2 - \frac{1}{4}b^4 - 4ab + \frac{1}{2}b^3 + 2a - \frac{1}{4}b^2)(b^2 + 4a - b)\label{bp_1(a,b;2,2)}.
\end{align}
    We respectively present the graphs of  \eqref{bp_1(a,b;3,1)} and \eqref{bp_1(a,b;2,2)} in Fig. \ref{fig:2}. 
\begin{figure}[H]
\centering
\begin{minipage}{0.49\linewidth}
    \centering
    \includegraphics[width=0.8\linewidth]{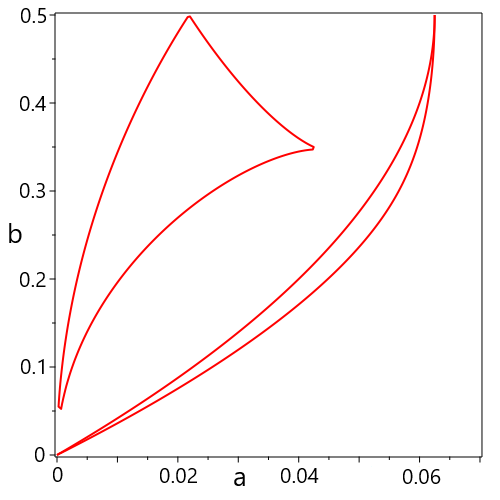}
    \subcaption{}\label{fig:2a}
\end{minipage}
\begin{minipage}{0.49\linewidth}
    \centering
    \includegraphics[width=0.8\linewidth]{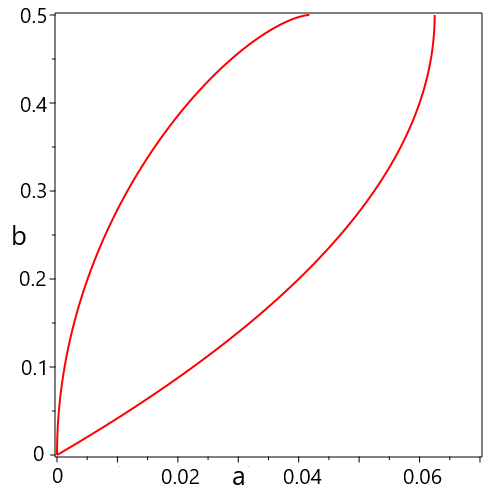}
    \subcaption{}\label{fig:2b}
\end{minipage}
\caption{(a) For $n_1=3,\;n_2=1$, we plot the curve generated by ${\tt bp}_1(a,b;3,1)$. (b) For $n_1=n_2=2$, we plot the curve generated by ${\tt bp}_1(a,b;2,2)$. }
\label{fig:2}
\end{figure}
\begin{remark}\label{remark1}
In this remark, we explain why the range of $a$ presented in Fig. \ref{fig:2b} is $(0,0.07)$ instead of $(0,+\infty)$.
Similarly, one can understand why the range of $a$ in Fig. \ref{fig:2a} is $(0,0.07)$. Note that by \eqref{bp_1(a,b;2,2)}, 
\begin{align}
    {\tt bp}_1(a,b;2,2) &= 2ab(b-\frac{1}{2})(2a + \frac{b}{2})(2a - \frac{b}{2} + \frac{1}{2})g_1(a, b)g_2(a, b),
\end{align}
    where
\begin{align}
    g_1(a, b) &:= 216a^3 - 36a^2b^2 + 2ab^4 + 36a^2b - 4ab^3 - 36a^2 + 6ab^2 - \frac{1}{4}b^4 - 4ab + \frac{1}{2}b^3 + 2a \notag\\&- \frac{1}{4}b^2  \\
    g_2(a, b) &:= b^2 + 4a - b.
\end{align}
Note that \( \frac{\partial g_2}{\partial a}=4 \), which indicates that \( g_2 \) is increasing with respect to \( a \). 
 Below, we prove that $g_1$ is also increasing with respect to \( a \). In fact, we can compute that
\begin{align*}
    \frac{\partial g_1}{\partial a} &= 2(324a^2 - 36ab^2 + 36ab - 36a + b^4 - 2b^3 + 3b^2 - 2b + 1), \\
    \frac{\partial^2 g_1}{\partial a^2} &=648a -36( b-\frac{1}{2})^2-27.
\end{align*}
Obviously, \( \frac{\partial^2 g_1}{\partial a^2} \) is increasing with respect to \( a \). Note that for $a=0$, \( \frac{\partial^2 g_1}{\partial a^2} \) is negative, and when $a$ is large enough, \( \frac{\partial^2 g_1}{\partial a^2} \) is positive. So, \( \frac{\partial g_1}{\partial a} \) is first decreasing and then increasing with respect to \( a \). We solve $a$ from \( \frac{\partial^2 g_1}{\partial a^2}=0 \), and we get 
\begin{align}
    a = \frac{b^2}{18} - \frac{b}{18} + \frac{1}{18}.\label{eq:case a}
\end{align}
Substituting \eqref{eq:case a} into \( \frac{\partial g_1}{\partial a} \), we find that \( \frac{\partial g_1}{\partial a}=0 \). So, \( \frac{\partial g_1}{\partial a} \) is always non-negative for any $a\in {\mathbb R}_{>0}$ and for any $b\in (0,0.5)$. Hence, \( g_1 \) is increasing with respect to \( a \). 
By the implicit function theorem, we see that  $g_1(a, b)=0$ and $g_2(a,b)=0$ respectively define two implicit functions, say $a=a_1(b)$, and $a=a_2(b)$. It is directly to check that both are increasing functions.    
Note that for $b=0.5$, the only real root of $g_1(a,b)$ is $a\approx 0.04$, and the only real root of $g_2(a,b)$ is $a\approx 0.0625$. Therefore, for $a\ge 0.07$, ${\tt bp}_1(a,b;2,2)=0$ has no real solutions for any $b\in (0,0.5)$. 
\end{remark}

\subsubsection{Case 2: Steady State Consisting of Three Numbers}\label{section:Solution consists of three numbers}
In this section, we deal with \ref{Case_2}. 
Suppose $x$ is a steady state of the system \eqref{eq:case1}. Note again that $x$ is also a real solution of the system \eqref{eq:SAS}. 
According to the hypothesis of \ref{Case_2}, 
we assume that the coordinates of $x$ consist of three positive numbers $y$, $z$ and $w$, and they appear in $x$ respectively $n_1$ , $n_2$ and $n_3$ times ($n_1+n_2+n_3=n,\;n\ge 3$). We substitute $x_1=\cdots=x_{n_1}=y,\;x_{n_1+1}=\cdots=x_{n_1+n_2}=z,\;x_{n_1+n_2+1}=\cdots=x_n=w$ into the first $n$ algebraic equations listed in the system \eqref{eq:SAS}, and we get
    \begin{align}
\mathscr{G}_{21}(n_1,n_2,n_3)&:=y(1 - y)(y - b) - nay + a(n_1y + n_2z + n_3w)=0,\label{eq:f21}\\
\mathscr{G}_{22}(n_1,n_2,n_3)&:=z(1 - z)(z - b) - naz + a(n_1y + n_2z + n_3w)=0,\label{eq:f22}\\
\mathscr{G}_{23}(n_1,n_2,n_3)&:=w(1 - w)(w - b) - naw + a(n_1y + n_2z + n_3w)=0,\label{eq:f23}
\end{align}
where 
\begin{align}\label{eq:con23}
y,z,w,a\in {\mathbb R}_{\geq 0},\;b\in [0,0.5]
\end{align}
We call the above system ${\tt \mathscr{G}}_2(n_1, n_2,n_3)$. Using ${\tt Maple}$, we compute the border polynomial of the system ${\tt \mathscr{G}}_2(n_1, n_2,n_3)$ with the constraints \eqref{eq:con23}, and we obtain
\begin{align}
&{\tt bp}_2(a,b;n_1,n_2,n_3):=abn_1n_2n_3(b - \frac{1}{2})(n_1 - n_2)(16a^3n_1^3 + \frac{3}{4}a^3n_2^2n_3 +\cdots -\frac{1}{2}a^2b^2n_2n_3)\notag\\&\cdots (a^3n_1^3 - 15a^3n_1^2n_2 +\cdots -b^2). \label{bp_2(n_1,n_2,n_3)} 
\end{align}
We provide a ${\tt Maple}$ file \footnote{See: https://github.com/songkuo-ux/Allee-Effect/blob/master/3.2.2.mw} containing the full expression of ${\tt bp}_2(a,b;n_1,n_2,n_3)$ for the readers to check the computation presented in this section.
Note that if $n_1\neq n_2$ and $n_2=n_3$, then
\begin{align}
&{\tt bp}_2(a,b;n_1,n_2,n_2)=abn_1n_2(b - \frac{1}{2})(b + 1)(27a^3n_2^3 - 9a^2b^2n_2^2+\cdots -\frac{1}{4}b^2)\cdots(a^3n_1^3\notag\\& - 12a^3n_1^2n_2 +\cdots-b^2).\label{bp_2(n_1,n_3,n_3)}
\end{align}
Note that if $n_1= n_2= n_3$, then
\begin{align}
&{\tt bp}_2(a,b;n_1,n_1,n_1)=abn_1(b-\frac{1}{2})(b + 1)(27a^3n_1^3 - 9a^2b^2n_1^2 + ab^4n_1 + 9a^2bn_1^2 - 2ab^3n_1\notag\\& - 9a^2n_1^2 + 3ab^2n_1 - \frac{1}{4}b^4 - 2abn_1 + \frac{1}{2}b^3 + an_1 -\frac{1}{4}b^2).\label{bp_2(n_1,n_1,n_1)}
\end{align}
For instance, for $n_1=3,\;n_2=2,\;n_3=1$, the polynomial \eqref{bp_2(n_1,n_2,n_3)} becomes \eqref{eq:bp_2(a,b;4,3,2)}. For $n_1=4,\;n_2=n_3=1$, the polynomial \eqref{bp_2(n_1,n_3,n_3)} becomes \eqref{eq:bp_2(a,b;4,4,1)}, and for $n_1=n_2=n_3=2$, the polynomial \eqref{bp_2(n_1,n_1,n_1)} becomes \eqref{eq:bp_2(a,b;3,3,3)}.
\begin{align}
    &{\tt bp}_2(a,b;3,2,1) = 6ab(b-\frac{1}{2})(3a+\frac{1323}{2}a^3 - \frac{315}{4}a^2  +\cdots +3ab^4)\cdots (8ab^4 - 144a^2b^2 \notag\\& +\cdots+8a), \label{eq:bp_2(a,b;4,3,2)} \\
    &{\tt bp}_2(a,b;4,1,1) = 12ab(b-\frac{1}{2})(b + 1)(27a^3 - 9a^2b^2 +\cdots- \frac{1}{4}b^2)\cdots(4ab^4 - 27a^2b^2 +\notag\\&\cdots+ 4a), \label{eq:bp_2(a,b;4,4,1)} \\
    &{\tt bp}_2(a,b;2,2,2) = 2ab(b-\frac{1}{2})(b + 1)(216a^3 - 36a^2b^2 + 2ab^4 + 36a^2b - 4ab^3 - 36a^2 \notag\\&+ 6ab^2 - \frac{1}{4}b^4 - 4ab + \frac{1}{2}b^3 + 2a - \frac{1}{4}b^2). \label{eq:bp_2(a,b;3,3,3)}
    \end{align}
We respectively present the graphs of the above polynomials \eqref{eq:bp_2(a,b;4,3,2)}, \eqref{eq:bp_2(a,b;4,4,1)} and  \eqref{eq:bp_2(a,b;3,3,3)} in Fig. \ref{fig:3}.
\begin{remark}
The method of proving the ranges of \( a \) plotted in Fig.\ref{fig:3} is similar to that presented in Remark \ref{remark1}.   
\end{remark}
\begin{figure}[H]
    \centering
    \begin{minipage}{0.3\linewidth}
        \centering
        \includegraphics[width=0.9\linewidth]{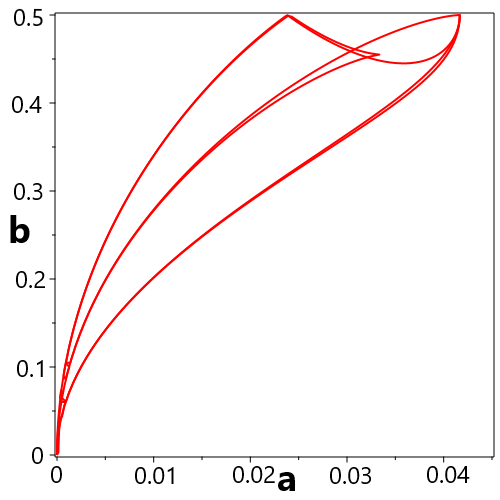}
        \subcaption{}\label{imp1}
    \end{minipage}
    \begin{minipage}{0.3\linewidth}
        \centering
        \includegraphics[width=0.9\linewidth]{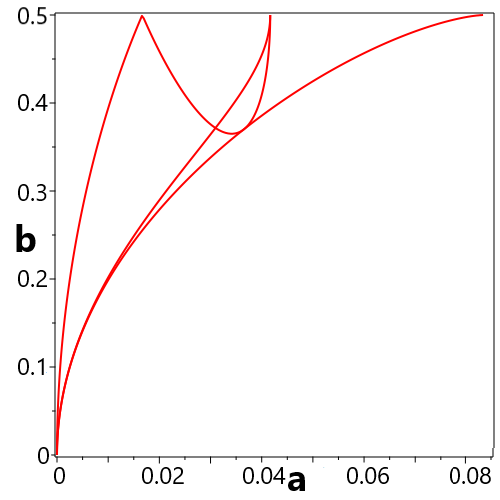}
        \subcaption{}\label{imp2}
    \end{minipage}
    \begin{minipage}{0.3\linewidth}
        \centering
        \includegraphics[width=0.9\linewidth]{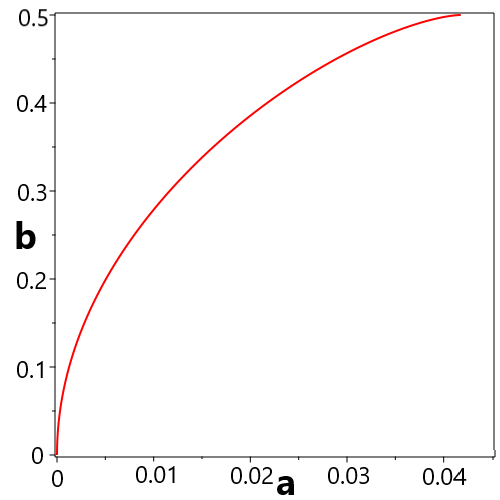}
        \subcaption{}\label{imp3}
    \end{minipage}
    \caption{(a) $n_1=3,\;n_2=2,\;n_3=1$, we plot the curve generated by ${\tt bp}_2(a,b;3,2,1)$. (b) For $n_1=4,\;n_2=n_3=1$, we plot the curve generated by ${\tt bp}_2(a,b;4,1,1)$.  (c) For $n_1=n_2=n_3=2$, we plot the curve generated by ${\tt bp}_2(a,b;2,2,2)$.}
    \label{fig:3}
\end{figure}
\subsubsection{Computing Border Polynomials}\label{section:Conclusion}
From the discussion presented in the previous two sections, we can conclude the following theorem, by which one can easily derive a border polynomial of the system \eqref{eq:SAS} for any dimension $n\in {\mathbb{N}_{+}}$.  
\begin{theorem}\label{thm:discriminant}
    For any $n \in {\mathbb{N}_{+}}$, a border polynomial of the system \eqref{eq:SAS} can be written as:
\begin{align}
   {\tt bp}(a,b;n) \;:=\; \prod_{\substack{n_1+n_2=n \\ n_1\geq n_2\textgreater 0}}^{}{{\tt bp}_1(a,b;n_1,n_2)}\prod_{\substack{n_1+n_2+n_3=n\\ n_1\geq n_2\geq n_3\textgreater 0}}^{}{{\tt bp}_2(a,b;n_1,n_2,n_3)}.\label{eq:theorem_2}
\end{align}
\end{theorem}
\begin{remark}
    For any positive integer \( n \) $(n\ge 3)$, there are \( n-2 \)  ways to partition it into three positive integers \( n_1 \), \( n_2 \) and \( n_3 \) satisfying $n_1\geq n_2\geq n_3$, and for any positive integer \( n \) $(n\ge 2)$, there are \(\left\lfloor \frac{n}{2} \right\rfloor\)  ways to partition it into two positive integers \( n_1 \) and \( n_2 \) satisfying $n_1\geq n_2$.
\end{remark}
\begin{example}\label{remark3}
For instance, for $n=4$, the border polynomial \eqref{eq:theorem_2} becomes 
\begin{align}
{\tt bp}(a, b;4)={\tt bp}_1(a, b;3,1){\tt bp}_1(a, b;2,2){\tt bp}_2(a, b;2,1,1)\label{eq:bp4},
\end{align}
where ${\tt bp}_1(a,b;3,1)$ and ${\tt bp}_1(a,b;2,2)$ are given in \eqref{bp_1(a,b;3,1)} and \eqref{bp_1(a,b;2,2)}, and by \eqref{bp_2(n_1,n_3,n_3)}, we can easily get 
\begin{align}
&{\tt bp}_2(a, b;2,1,1)=2ab(b-\frac{1}{2})(b + 1)( 27a^3 +\cdots- \frac{1}{4}b^2)\cdots(4ab^4 +\cdots + 4a).\label{bp2(2,1,1)}
\end{align}
\end{example}
\begin{remark}
    Comparing  to other methods for computing border polynomials or discriminant varieties  in ${\tt Maple}$, applying Theorem \ref{thm:discriminant} is much  more efficient for larger $n$, see Table \ref{table_1}.
    \end{remark}

\begin{table}[htbp]
\caption{Computational timings  for computing border polynomials}\label{tab2}
\begin{tabular*}{\textwidth}{@{\extracolsep\fill}cccc}
\toprule%
& \multicolumn{3}{@{}c@{}}{Timings\footnotemark[1]} \\\cmidrule{2-4}%
Dimension\footnotemark[2] & {\tt DiscriminantVariety} & {\tt BorderPolynomial} & Theorem $2$ \\
\midrule
$n=3$  & $\ge 2h$ & $0.58s$ & $0.062s$\\
$n=4$  & $\ge 2h$ & $2.4s$ &  $0.094s$\\
$n=5$  & $\ge 2h$ & $18s$ & $0.14s$\\
$n=6$  & $\ge 2h$ & $174s$ & $0.28s$ \\
$n=7$  & $\ge 2h$ & $\ge 2h$ & $0.69s$\\
$n=100$  & $\ge 2h$ & $\ge 2h$ & $9.8s$ \\
\botrule
\label{table_1}
\end{tabular*}
\footnotetext   {Note: We run the experiments by a 2.60 GHz Intel Core i7-9750H processor (8GB total memory) under Windows 10. Using {\tt Maple} command {\tt DiscriminantVariety}, we can not compute the discriminant variety of the system \eqref{eq:SAS} within $2$ hours for any $n\geq 3$. Using {\tt Maple} command {\tt BorderPolynomial}, we can compute the border polynomial of the system \eqref{eq:SAS} for $3\leq n\leq 6$ in a reasonable time. Applying Theorem \ref{thm:discriminant}, we can compute the border polynomial for pretty large $n$ such as $n=100$ in a short time.}

\footnotetext[1]{``Timings" means the computational time for completing the computation.}
\footnotetext[2]{``Dimension" means the number of coordinates of $x$ in the system \eqref{eq:SAS}.}
\end{table}

\begin{remark}
We remark that the border polynomial computed by the {\tt Maple} command {\tt BorderPolynomial} may give some redundant factors (i.e., the curve generated by the border polynomial may have some extra branches).  Comparing Fig.\ref{4.a} and Fig. \ref{4.b}, the blue curves plotted in Fig.\ref{4.a} are generated by those redundant factors.
\begin{figure}[H]
    \centering
    \begin{minipage}{0.49\linewidth}
        \centering
            \includegraphics[width=5.26cm, height=5.3cm]{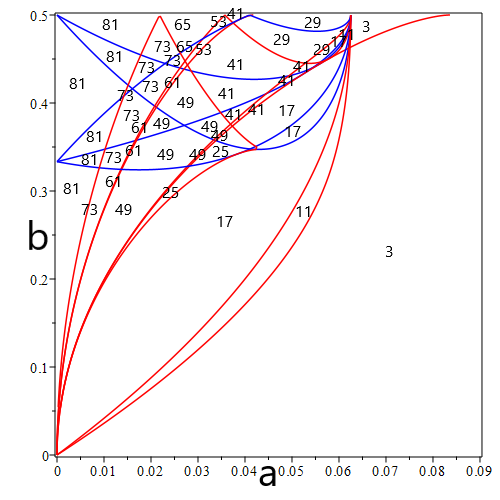}
            \subcaption{}\label{4.a}
    \end{minipage}
    \begin{minipage}{0.35\linewidth}
        \centering
        \includegraphics[width=5.4cm, height=5.3cm]{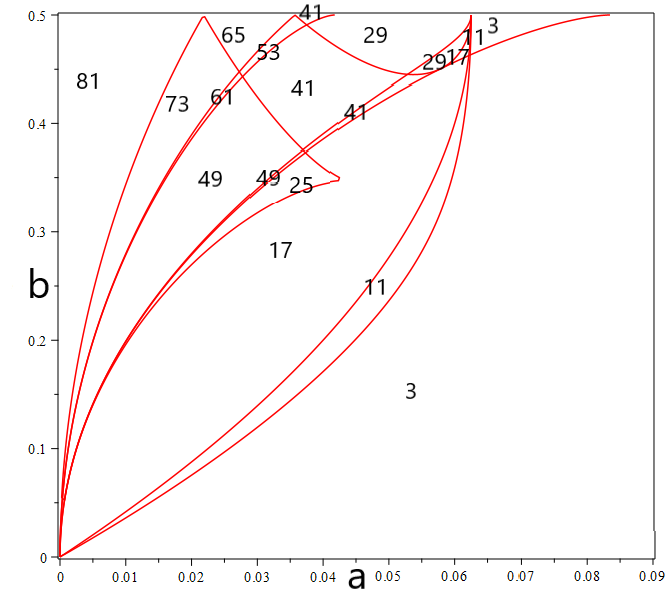}
        \subcaption{}\label{4.b}
    \end{minipage}
    \caption{(a) For $n=4$, we plot the curve generated by the border polynomial computed by the {\tt Maple} command {\tt BorderPolynomial}. (b) For $n=4$, we plot the curve generated by the border polynomial ${\tt bp}(a, b; 4)$ presented in \eqref{eq:bp4}. Over each open connected region, we give the number of real solutions of the system \eqref{eq:SAS}.}
    \end{figure}
\end{remark}
\subsection{Algorithm}\label{section:Solving Scheme-Region solution}
The hypersurface generated by the border polynomial \eqref{eq:theorem_2} divides the region  ${\mathbb R}_{\ge 0}\times [0,0.5]$ into finitely many open connected components. By Definition \ref{def:bp}, the number of real solutions of the system \eqref{eq:SAS}  (i.e., the number of steady states of the system \eqref{eq:case1}) is a constant over each open connected component (see Fig. \ref{4.a} and Fig. \ref{4.b}). 
In this section, we will show how to  compute the number of  steady states of the system \eqref{eq:case1} in each component by the following steps.  
\begin{enumerate}[label={\textbf{Step \arabic*}},ref=Step \arabic*]
    \item \label{Step_1}
   For any fixed $n \in {\mathbb{N}_{+}}$, according to Theorem \ref{thm:discriminant}, we compute the border polynomial ${\tt bp}(a, b;n)$  shown in \eqref{eq:theorem_2}.
    \item We apply cylindrical algebraic decomposition (CAD) to ${\tt bp}(a, b;n)$, and we get finitely many sample points, denoted by $(a_1,b_1),\ldots,(a_s,b_s)$ (here, by ``sample point" we mean for any open connected component determined by ${\tt bp}(a, b;n)\neq 0$, there exists $i\in \{1, \ldots, s\}$ such that the point $(a_i, b_i)$ is located in this component).  
    \item For each sample point $(a_i,b_i)\;(1\le i \le s)$, we compute the number of  steady states of the system \eqref{eq:case1} at the point $(a_i,b_i)$ by the following steps.
    \begin{enumerate}
        \item[\bf Step 3.1] We transform  the $n$-dimensional  system \eqref{eq:SAS} into several  two-dimensional and three dimensional systems. Recall that these systems are called $\mathscr{G}_1(n_1, n_2)$  (see \eqref{eq:f11}--\eqref{eq:f12}) for all $n_1$ and $n_2$ satisfying $n_1+n_2=n$ and  $n_1\geq n_2\textgreater 0$, and $\mathscr{G}_2(n_1,n_2,n_3)$ (see \eqref{eq:f21}--\eqref{eq:f23}) for all $n_1$, $n_2$ and $n_3$ satisfying $n_1+n_2+n_3=n$ and  $n_1\geq n_2\geq n_3\textgreater 0$. 
Notice that 
        according to Remark \ref{remark3}, there are $\left\lfloor \frac{n}{2} \right\rfloor$ two-dimensional  systems  and $n-2$ three-dimensional systems. 
        \item[\bf Step 3.2] For  $a=a_i$ and $b=b_i$, we compute the numbers of positive solutions for the algebraic systems  $\mathscr{G}_1(n_1, n_2)$  and $\mathscr{G}_2(n_1,n_2,n_3)$, denoted by $c_1(n_1,n_2)$ and $c_2(n_1,n_2,n_3)$. According to Theorem \ref{theorem3}, the number of steady states for the system \eqref{eq:case1} is given by the formula
        \eqref{eq:sum up}.
    \end{enumerate}
\end{enumerate}
\begin{theorem}\label{theorem3}
Given a positive integer  $n \in {\mathbb{N}_{+}}$ $(n\ge 3)$,  for any $a\in {\mathbb R}_{>0}$ and for any $b\in [0,0.5]$, if  the numbers of positive solutions for the systems $\mathscr{G}_1(n_1, n_2)$ (see \eqref{eq:f11}--\eqref{eq:f12}) and $\mathscr{G}_2(n_1,n_2,n_3)$ (see \eqref{eq:f21}--\eqref{eq:f23}) are $c_1(n_1,n_2)$ and $c_2(n_1,n_2,n_3)$, then the number of steady states of the system \eqref{eq:case1} is 
\begin{align}
3+\sum_{\substack{n_1+n_2=n\\ n_1\geq n_2\textgreater 0}}^{}{c_1(n_1,n_2)\binom{n}{n_1}}+\sum_{\substack{n_1+n_2+n_3=n\\ n_1\geq n_2\geq n_3\textgreater 0}}^{}{c_2(n_1,n_2,n_3)\binom{n}{n_1}\binom{n-n_1}{n_2}}\label{eq:sum up}.
\end{align}
\end{theorem}
\begin{proof}
 By Theorem \ref{theorem 1},   the coordinates of any steady state $x=(x_1,\cdots,x_n)\in {\mathbb R}_{>0}^n$ consist of at most three distinct positive numbers. If these coordinates are the same, then there are always three trivial steady states according to \ref{Case_0}. If these coordinates consist of two distinct numbers $y$ and $z$, then
 $(y,z)$ is a solution of the system $\mathscr{G}_1(n_1, n_2)$. Notice that there are $\binom{n}{n_1}$ kinds of solution vector in ${\mathbb R}^n$
 satisfying that $y$ and $z$ appear in $x$ respectively $n_1$ and $n_2$ times. If these coordinates consist of three distinct numbers $y$, $z$ and $w$, then
 $(y,z,w)$ is a solution of the system $\mathscr{G}_2(n_1, n_2, n_3)$. Notice that there are $\binom{n}{n_1}\binom{n-n_1}{n_2}$ kinds of solution vector in ${\mathbb R}^n$
 satisfying that $y$, $z$ and $w$ appear in $x$ respectively $n_1$, $n_2$ and $n_3$ times.
  So, the number of steady states for the system \eqref{eq:case1} is given by 
  \eqref{eq:sum up}.
\end{proof}
\subsection{Computational Results for High Dimensional Systems}\label{section:Examples}
We illustrate how to carry out the algorithm presented in Section \ref{section:Solving Scheme-Region solution} by the following example, which answers Problem \ref{Problem} for $n=4$.   
\begin{example}\label{example 1}
 For $n = 4$, we compute the number of the steady states of the system \eqref{eq:case1} for any generic $(a,b)\in {\mathbb R}_{\ge 0}\times [0,0.5]$. 
 \begin{enumerate}
    \item [\textbf{Step 1}]
For $n=4$, we write done the border polynomial given in \eqref{eq:theorem_2}:
\begin{align}
    {\tt bp}(a,b;4)={\tt bp}_2(a,b;2,1,1){\tt bp}_1(a,b;3,1){\tt bp}_1(a,b;2,2),\label{eq:exzample_disc}
\end{align}
where ${\tt bp}_1(a,b;3,1)$, ${\tt bp}_1(a,b;2,2)$  and ${\tt bp}_2(a,b;2,1,1)$ are given in \eqref{bp_1(a,b;3,1)}, \eqref{bp_1(a,b;2,2)}, and \eqref{bp2(2,1,1)}.
    
\item [\textbf{Step 2}] We apply CAD \cite{bib_CAD} to ${\tt bp}(a,b;4)$, and we get $469$ sample points denoted by $(a_1,b_1),\ldots,(a_{469},b_{469})$. This step is implemented by following commands in {\tt Maple}:
              \begin{figure}[H]
                    \centering
                 \includegraphics[width=8.26cm, height=1.85cm]{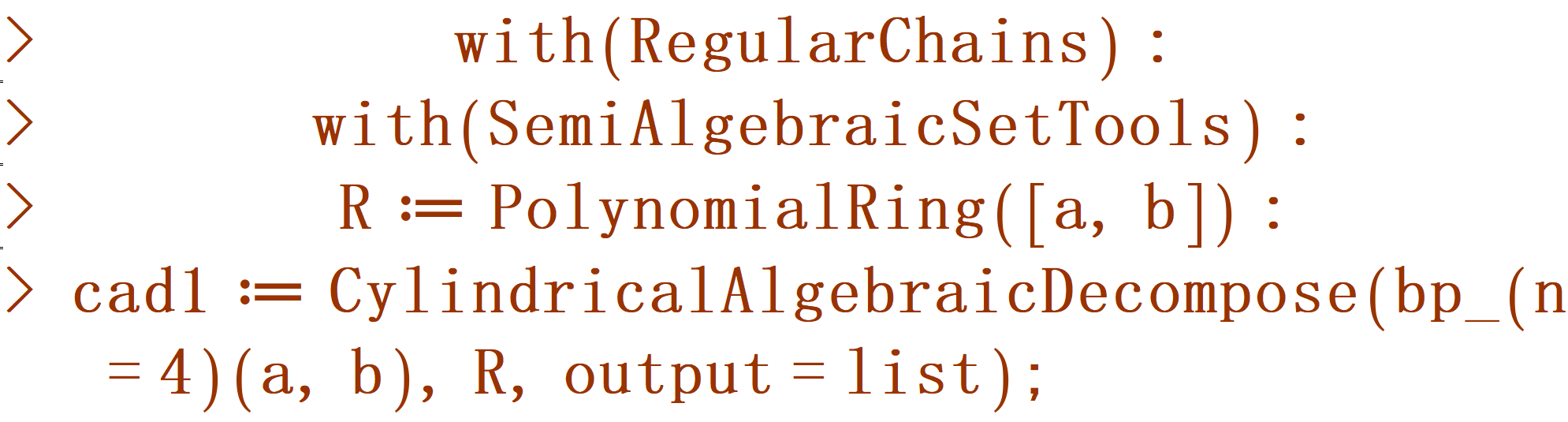}
            \end{figure}
\item [\textbf{Step 3}] 
For each sample point $(a_i,b_i)\;(1\le i \le 469)$, we compute the number of  steady states of the system \eqref{eq:case1} at the point $(a_i,b_i)$. For instance, we show the computational steps below for the first sample point $(a_1,b_1)=(\frac{1319}{1048576}, \frac{363843}{2097152})$.
\begin{enumerate}
    \item[\bf Step 3.1] For $n=4$, we transform the $n$-dimensional system \eqref{eq:SAS} into $2$ two-dimensional systems and
$1$ three-dimensional system, denoted by $\mathscr{G}_1(2,2)$ (see \eqref{eq:case16}), $\mathscr{G}_1(3,1)$ (see \eqref{eq:case17}) and $\mathscr{G}_2(2,1,1)$ (see \eqref{eq:case18}).
    \begin{enumerate}
        \item Assume that  the coordinates of $x$ consist of two positive numbers $y$ and $z$, and $y$ and $z$ appear in $x$ respectively $2$ and $2$ times. Suppose that  $x_1=x_{2}=y$, $x_{3}=x_4=z$. We substitute $x_1=x_{2}=y,\;x_{3}=x_4=z$ into the system \eqref{eq:SAS}, and we get $\mathscr{G}_1(2,2)$ (or, one can directly substitute $n_1=n_2=2$ into the systems \eqref{eq:f11}--\eqref{eq:f12}):
            \begin{equation}
            \begin{aligned}
                y(1 - y)(y - b) - 4ay + a(2y+2z)&=0,\\
                z(1 - z)(z - b) - 4az + a(2y+2z)&=0.\label{eq:case16}
            \end{aligned}
            \end{equation}
        \item  Assume that the coordinates of $x$ consist of two positive numbers $y$ and $z$, and $y$ and $z$ appear in $x$ respectively $3$ and $1$ times. Suppose that  $x_1=x_2=x_{3}=y$, $x_4=z$. We substitute $x_1=x_2=x_{3}=y,\;x_4=z$ into the system \eqref{eq:SAS}, and we get $\mathscr{G}_1(3,1)$ (or, one can directly substitute $n_1=3$ and $n_2=1$ into the systems \eqref{eq:f11}--\eqref{eq:f12}):
            \begin{equation}
            \begin{aligned}
                y(1 - y)(y - b) - 4ay + a(3y+z)&=0,\\
                z(1 - z)(z - b) - 4az + a(3y+z)&=0.\label{eq:case17}
            \end{aligned}
            \end{equation}
        \item Assume that the coordinates of $x$ consist of three positive numbers $y$, $z$ and $w$, and $y$, $z$ and $w$ appear in $x$ respectively 
        $2$, $1$ and $1$ times. Suppose that $x_1=x_2=y$, $x_3=z$, $x_4=w$.
            We substitute $x_1=x_2=y,\;x_3=z,\;x_4=w$ into the system \eqref{eq:SAS}, and we get $\mathscr{G}_2(2,1,1)$ (or, one can directly substitute $n_1=2$ and $n_2=n_3=1$ into the systems \eqref{eq:f21}--\eqref{eq:f23}):
            \begin{equation}
            \begin{aligned}
                y(1 - y)(y - b) - 4ay + a(2y+z+w)&=0,\\
                z(1 - z)(z - b) - 4az + a(2y+z+w)&=0,\\
                w(1 - w)(w - b) - 4az + a(2y+z+w)&=0.\label{eq:case18}
            \end{aligned}
            \end{equation}
        \end{enumerate}
        \item[\bf Step 3.2] Using {\tt Maple} command {\tt Isolate} in the package {\tt RootFinding} \cite{bib3}, for the sample point $(a_1,b_1)=(\frac{1319}{1048576}, \frac{363843}{2097152})$,   we respectively  compute the numbers of positive solutions of the systems $\mathscr{G}_1(2,2)$ \eqref{eq:case16}, $\mathscr{G}_1(3,1)$ \eqref{eq:case17} and $\mathscr{G}_2(2,1,1)$ \eqref{eq:case18}, denoted by $c_1(2,2),\;c_1(3,1)$ and $c_2(2,1,1)$.
        \begin{enumerate}
            \item 
            For the system  $\mathscr{G}_1(2,2)$ \eqref{eq:case16}, we can get $c_1(2,2)=6$ by the following  {\tt Maple} command 
            \begin{figure}[H]
              \centering
             \includegraphics[width=8.26cm, height=1.65cm]{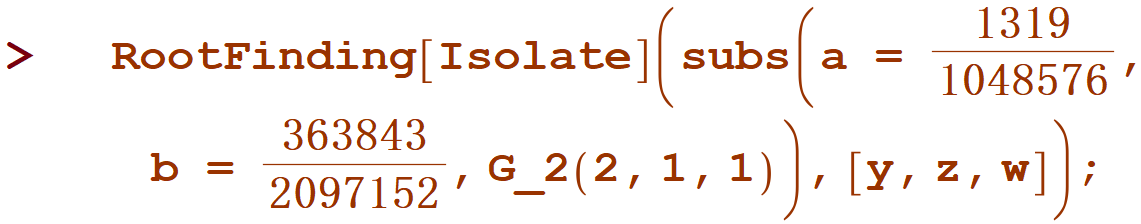}
                 \end{figure}
         \item
         For the system  $\mathscr{G}_1(3,1)$ \eqref{eq:case17}, we can get $c_1(3,1)=6$ by the following  {\tt Maple} command 
        \begin{figure}[H]
             \centering
             \includegraphics[width=8.26cm, height=1.65cm]{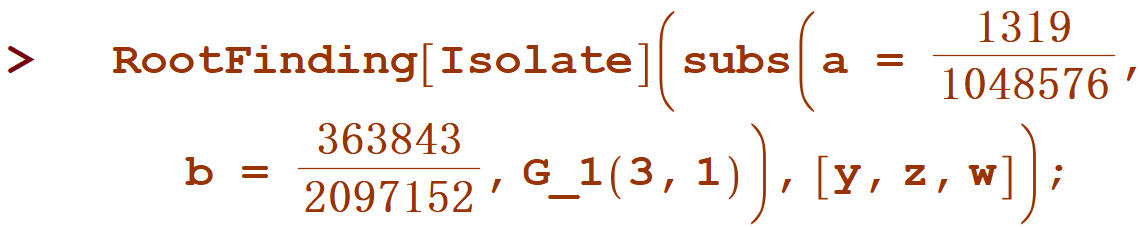}
         \end{figure}
        \item
        For the system $\mathscr{G}_2(2,1,1)$ \eqref{eq:case18}, we can get $c_2(2,1,1)=3$ by the following  {\tt Maple} command 
\begin{figure}[H]
        \centering
        \includegraphics[width=8.26cm, height=1.65cm]{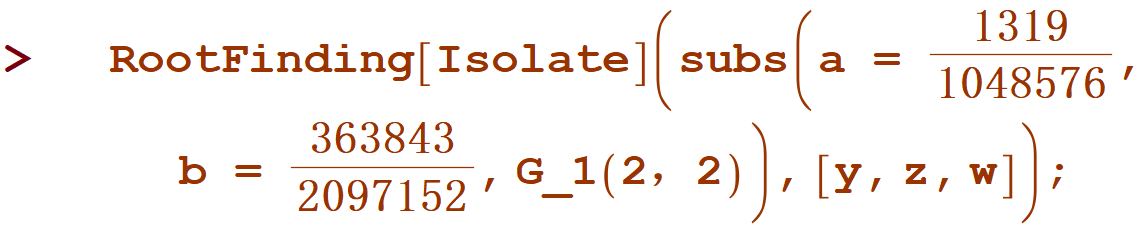}
        \end{figure}
\end{enumerate}
      By Theorem \ref{theorem3}, the number of steady states for the system \eqref{eq:case1} is 
    \begin{align}
        c_1(2,2)\binom{4}{2}+c_1(3,1)\binom{3}{1}+c_2(2,1,1)\binom{4}{2}\binom{2}{1}+3=81.
    \end{align}
    \end{enumerate}
   Similarly, we can get the number of steady states of the system \eqref{eq:case1} for the other sample points. And we plot the number of steady states over each open component determined by the border polynomial,   see Fig. \ref{Fig.4}. This graph is made by {\tt Paint 3D}.
\begin{figure}[htbp]
    \centering
    \includegraphics[height=6cm,width=14cm]{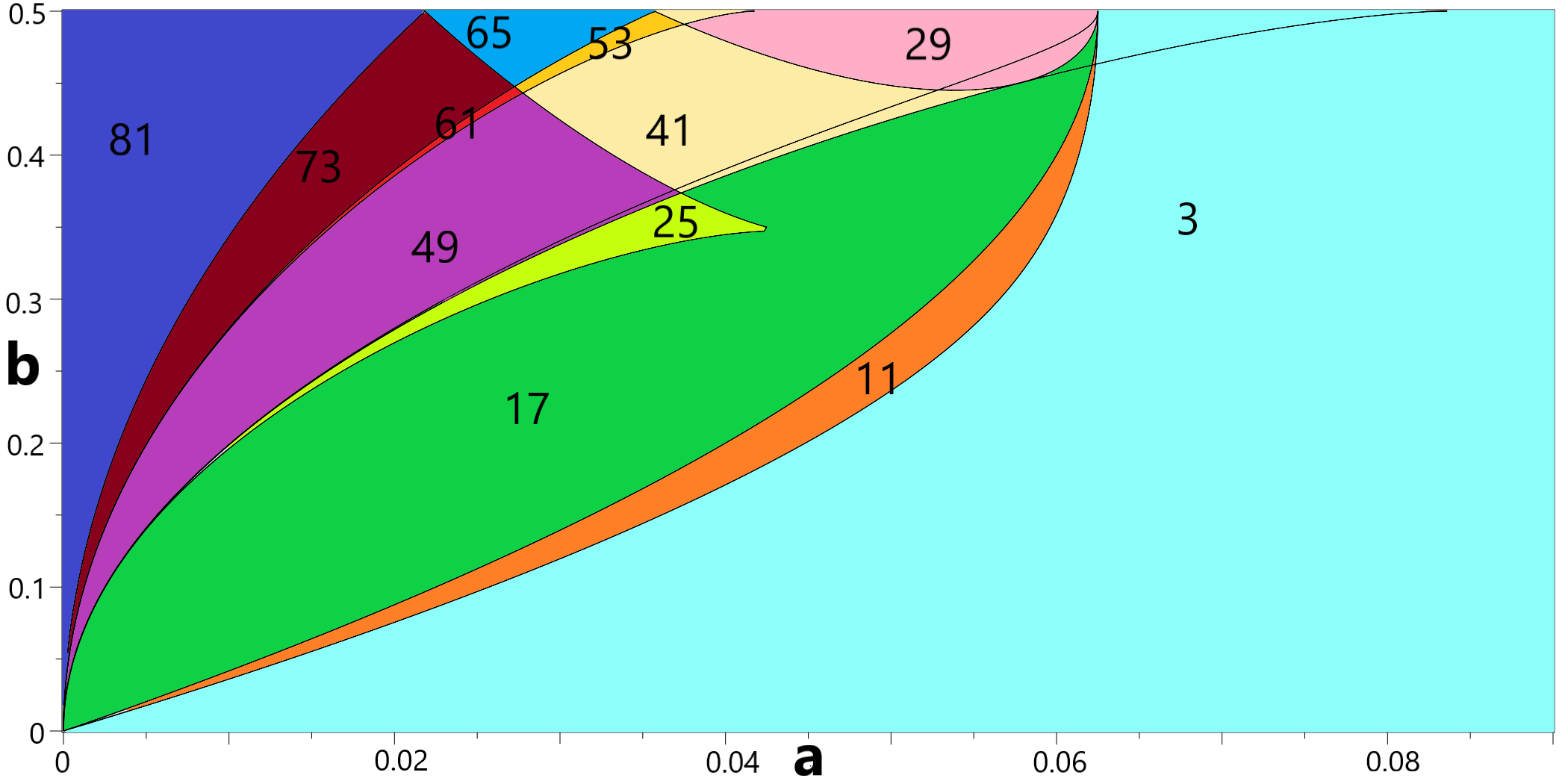}
    \caption{For $n=4$, we plot the (positive) real root classification of the steady-state system \eqref{eq:SAS}, which answers the steady state classification problem for the system \eqref{eq:case1}.}
    \label{Fig.4}
    \end{figure}
 \end{enumerate}
\end{example}
    Similarly to Example \ref{example 1}, we can compute and plot the (positive) real root classification of the steady-state system \eqref{eq:SAS} for $n=5,6,7$.
    We respectively present the graphs  for $n=5,6,7$  in Fig. \ref{the picture of $5$-dimensional system},  Fig. \ref{the picture of $6$-dimensional system} and Fig. \ref{the picture of $7$-dimensional system}. Above all, we have answered Problem 
    \ref{Problem} for $n=4,5, 6,7$. We provide a folder \footnote{see: https://github.com/songkuo-ux/Allee-Effect/blob/example} containing the computations presented in this section. 
\begin{figure}[H]
    \centering
    \includegraphics[height =6cm,width=14cm]{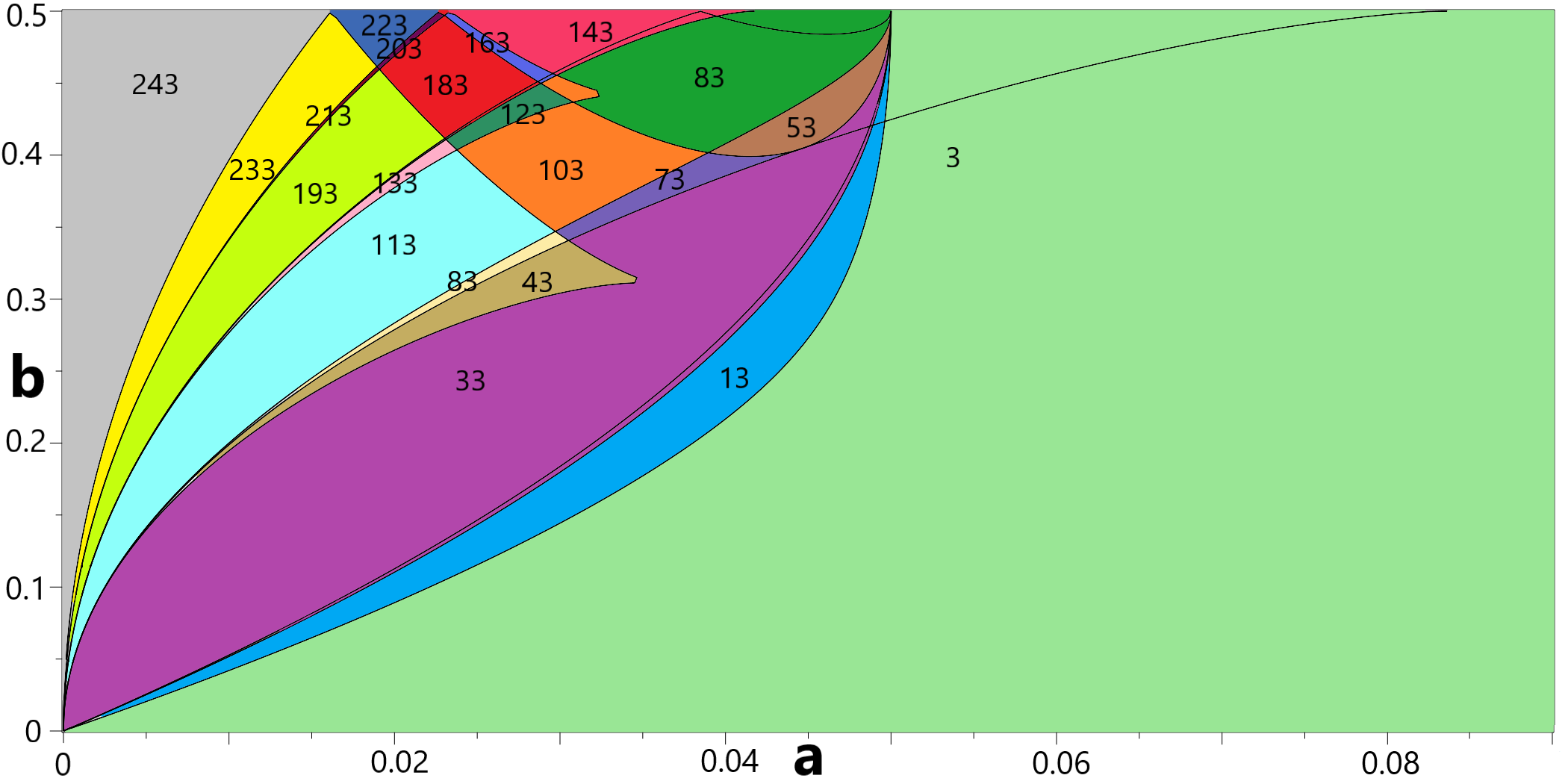}
    \caption{For $n=5$, we plot the (positive) real root classification of the steady-state system \eqref{eq:SAS}, which answers the steady state classification problem for the system \eqref{eq:case1}.}
    \label{the picture of $5$-dimensional system}
\end{figure}
\begin{figure}[H]
    \centering
    \includegraphics[height =6cm,width=14cm]{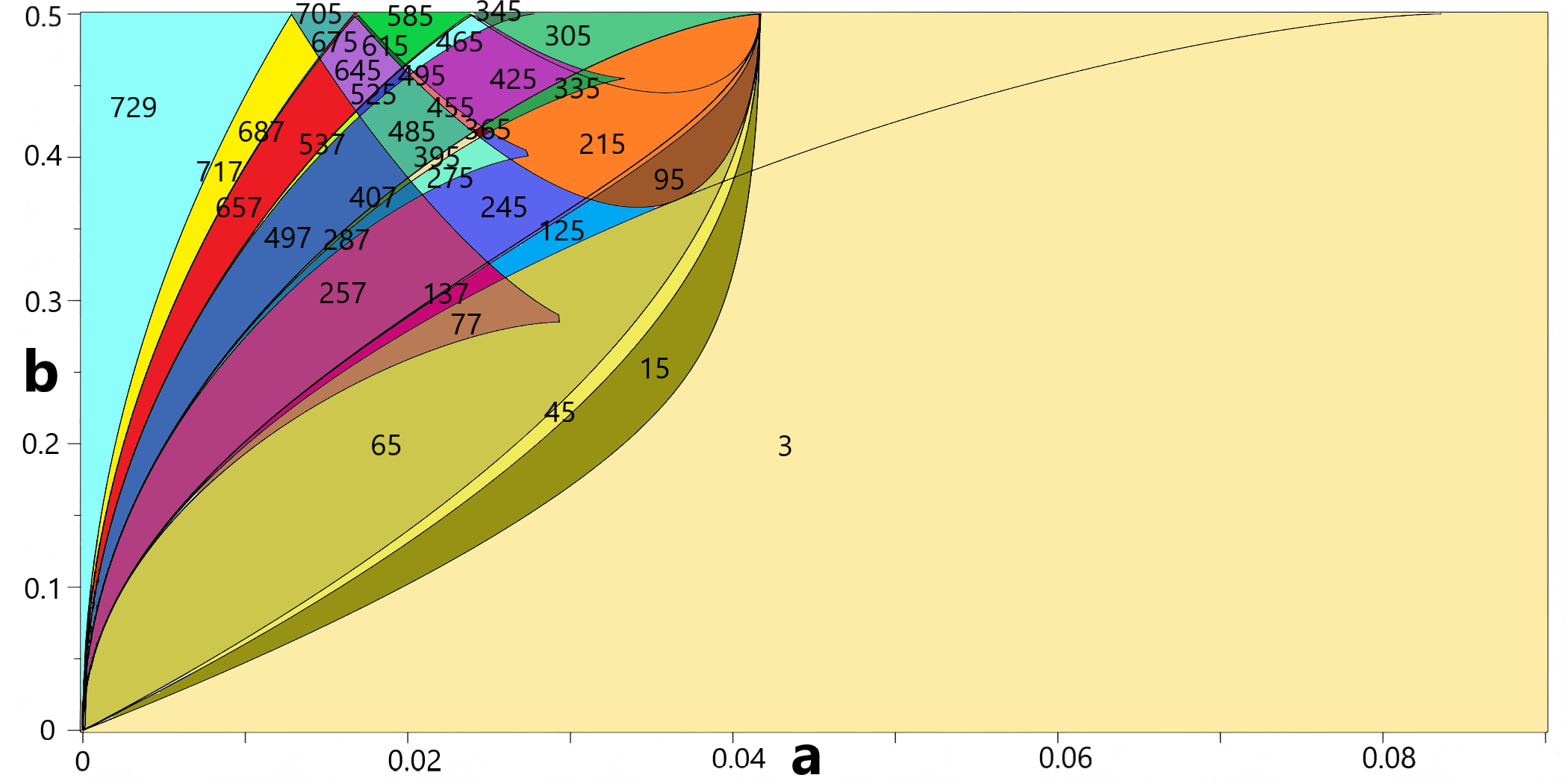}
    \caption{For $n=6$, we plot the (positive) real root classification of the steady-state system \eqref{eq:SAS}, which answers the steady state classification problem for the system \eqref{eq:case1}.}
    \label{the picture of $6$-dimensional system}
\end{figure}
\begin{figure}[H]
    \centering
    \includegraphics[height =6cm,width=14cm]{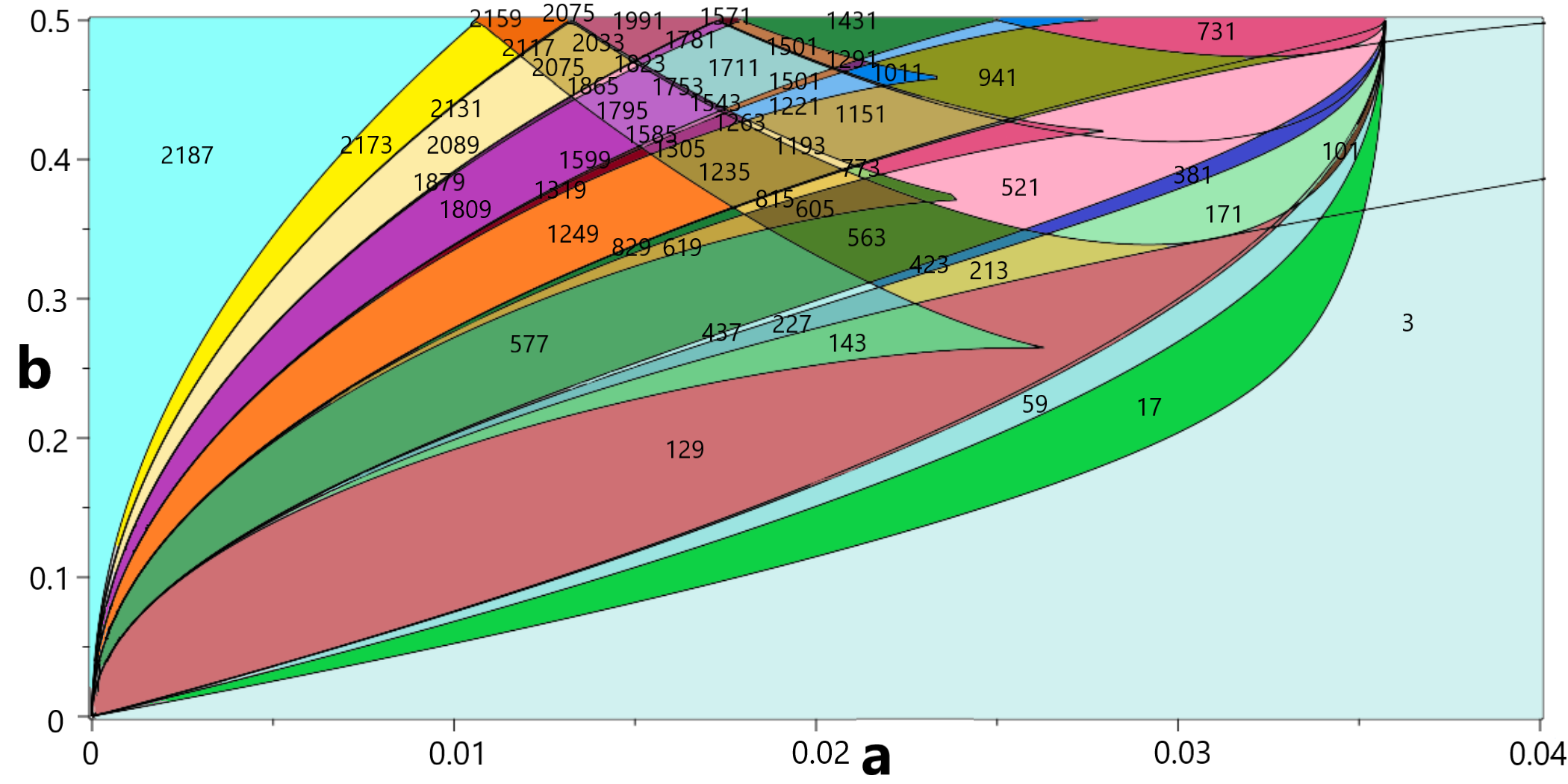}
    \caption{For $n=7$, we plot the (positive) real root classification of the steady-state system \eqref{eq:SAS}, which answers the steady state classification problem for the system \eqref{eq:case1}.}
    \label{the picture of $7$-dimensional system}
\end{figure}
\section{Conclusion and Discussion}\label{section:Conclusion and Discussion}
In this work, we solve Problem \ref{Problem} for $n=4,5,6,7$. 
According to our computational results, we can efficiently compute the border polynomials for pretty large $n$ (for instance, $n=100$). However, for $n>7$, it takes more than two days to carry out the CAD for the border polynomials. So, in the future, it is nice to  
study how to improve the efficiency for carrying out the CADs of border polynomials  for large $n$.
\bibliography{sn-article}


\begin{thebibliography}{45}
\ifx \bisbn   \undefined \def \bisbn  #1{ISBN #1}\fi
\ifx \binits  \undefined \def \binits#1{#1}\fi
\ifx \bauthor  \undefined \def \bauthor#1{#1}\fi
\ifx \batitle  \undefined \def \batitle#1{#1}\fi
\ifx \bjtitle  \undefined \def \bjtitle#1{#1}\fi
\ifx \bvolume  \undefined \def \bvolume#1{\textbf{#1}}\fi
\ifx \byear  \undefined \def \byear#1{#1}\fi
\ifx \bissue  \undefined \def \bissue#1{#1}\fi
\ifx \bfpage  \undefined \def \bfpage#1{#1}\fi
\ifx \blpage  \undefined \def \blpage #1{#1}\fi
\ifx \burl  \undefined \def \burl#1{\textsf{#1}}\fi
\ifx \doiurl  \undefined \def \doiurl#1{\url{https://doi.org/#1}}\fi
\ifx \betal  \undefined \def \betal{\textit{et al.}}\fi
\ifx \binstitute  \undefined \def \binstitute#1{#1}\fi
\ifx \binstitutionaled  \undefined \def \binstitutionaled#1{#1}\fi
\ifx \bctitle  \undefined \def \bctitle#1{#1}\fi
\ifx \beditor  \undefined \def \beditor#1{#1}\fi
\ifx \bpublisher  \undefined \def \bpublisher#1{#1}\fi
\ifx \bbtitle  \undefined \def \bbtitle#1{#1}\fi
\ifx \bedition  \undefined \def \bedition#1{#1}\fi
\ifx \bseriesno  \undefined \def \bseriesno#1{#1}\fi
\ifx \blocation  \undefined \def \blocation#1{#1}\fi
\ifx \bsertitle  \undefined \def \bsertitle#1{#1}\fi
\ifx \bsnm \undefined \def \bsnm#1{#1}\fi
\ifx \bsuffix \undefined \def \bsuffix#1{#1}\fi
\ifx \bparticle \undefined \def \bparticle#1{#1}\fi
\ifx \barticle \undefined \def \barticle#1{#1}\fi
\bibcommenthead
\ifx \bconfdate \undefined \def \bconfdate #1{#1}\fi
\ifx \botherref \undefined \def \botherref #1{#1}\fi
\ifx \url \undefined \def \url#1{\textsf{#1}}\fi
\ifx \bchapter \undefined \def \bchapter#1{#1}\fi
\ifx \bbook \undefined \def \bbook#1{#1}\fi
\ifx \bcomment \undefined \def \bcomment#1{#1}\fi
\ifx \oauthor \undefined \def \oauthor#1{#1}\fi
\ifx \citeauthoryear \undefined \def \citeauthoryear#1{#1}\fi
\ifx \endbibitem  \undefined \def \endbibitem {}\fi
\ifx \bconflocation  \undefined \def \bconflocation#1{#1}\fi
\ifx \arxivurl  \undefined \def \arxivurl#1{\textsf{#1}}\fi
\csname PreBibitemsHook\endcsname

\bibitem[\protect\citeauthoryear{Berec et~al.}{2007}]{bib_bib1}
\begin{barticle}
\bauthor{\bsnm{Berec}, \binits{L.}},
\bauthor{\bsnm{Angulo}, \binits{E.}},
\bauthor{\bsnm{Courchamp}, \binits{F.}}:
\batitle{Multiple allee effects and population management}.
\bjtitle{Trends in Ecology \& Evolution}
\bvolume{22}(\bissue{4}),
\bfpage{185}--\blpage{191}
(\byear{2007})
\doiurl{10.1016/j.tree.2006.12.002}
\end{barticle}
\endbibitem

\bibitem[\protect\citeauthoryear{Luís et~al.}{2009}]{bib_bib2}
\begin{barticle}
\bauthor{\bsnm{Luís}, \binits{R.}},
\bauthor{\bsnm{Elaydi}, \binits{S.}},
\bauthor{\bsnm{Oliveira}, \binits{H.}}:
\batitle{An economic economical model with allee effect}.
\bjtitle{Journal of Difference Equations and Applications}
\bvolume{15}(\bissue{8}),
\bfpage{877}--\blpage{894}
(\byear{2009})
\doiurl{10.1080/10236190802468920}
\end{barticle}
\endbibitem

\bibitem[\protect\citeauthoryear{Arim et~al.}{2022}]{bib_bib3}
\begin{barticle}
\bauthor{\bsnm{Arim}, \binits{M.}},
\bauthor{\bsnm{Herrera-Esposito}, \binits{D.}},
\bauthor{\bsnm{Bermolen}, \binits{P.}},
\bauthor{\bsnm{Cabana}},
\bauthor{\bsnm{Fariello}, \binits{M.I.}},
\bauthor{\bsnm{Lima}, \binits{M.}},
\bauthor{\bsnm{Romero}, \binits{H.}}:
\batitle{Contact tracing-induced allee effect in disease dynamics}.
\bjtitle{Journal of Theoretical Biology}
\bvolume{542},
\bfpage{111109}
(\byear{2022})
\doiurl{10.1016/j.jtbi.2022.111109}
\end{barticle}
\endbibitem

\bibitem[\protect\citeauthoryear{Zhu et~al.}{2024}]{bib_bib4}
\begin{barticle}
\bauthor{\bsnm{Zhu}, \binits{L.}},
\bauthor{\bsnm{Tao}, \binits{X.}},
\bauthor{\bsnm{Shen}, \binits{S.}}:
\batitle{Pattern dynamics in a reaction–diffusion predator–prey model with allee effect based on network and non-network environments}.
\bjtitle{Engineering Applications of Artificial Intelligence}
\bvolume{128},
\bfpage{107491}
(\byear{2024})
\doiurl{10.1016/j.engappai.2023.107491}
\end{barticle}
\endbibitem

\bibitem[\protect\citeauthoryear{R{\"o}st and Sadeghimanesh}{2021}]{bib_bib5}
\begin{barticle}
\bauthor{\bsnm{R{\"o}st}, \binits{G.}},
\bauthor{\bsnm{Sadeghimanesh}, \binits{A.}}:
\batitle{Exotic bifurcations in three connected populations with allee effect}.
\bjtitle{International Journal of Bifurcation and Chaos}
\bvolume{31}(\bissue{13}),
\bfpage{2150202}
(\byear{2021})
\doiurl{10.1142/S0218127421502023}
\end{barticle}
\endbibitem

\bibitem[\protect\citeauthoryear{Tarski}{1951}]{bib_bib6}
\begin{bbook}
\bauthor{\bsnm{Tarski}, \binits{A.}}:
\bbtitle{A Decision Method for Elementary Algebra and Geometry}.
\bpublisher{University of California Press},
\blocation{Berkeley and Los Angeles, California}
(\byear{1951})
\end{bbook}
\endbibitem

\bibitem[\protect\citeauthoryear{Collins}{1975}]{bib_bib7}
\begin{botherref}
\oauthor{\bsnm{Collins}, \binits{G.E.}}:
Quantifier elimination for real closed fields by cylindrical algebraic decompostion.
Paper presented at the Second GI Conference on Automata Theory and Formal Languages, University of Kaiserslautern, Germany, 20–23 May 1975
(1975)
\end{botherref}
\endbibitem

\bibitem[\protect\citeauthoryear{Arnon et~al.}{1988}]{bib_bib8}
\begin{barticle}
\bauthor{\bsnm{Arnon}, \binits{D.S.}},
\bauthor{\bsnm{Collins}, \binits{G.E.}},
\bauthor{\bsnm{McCallum}, \binits{S.}}:
\batitle{An adjacency algorithm for cylindrical algebraic decompositions of three-dimenslonal space}.
\bjtitle{Journal of Symbolic Computation}
\bvolume{5}(\bissue{1--2}),
\bfpage{163}--\blpage{187}
(\byear{1988})
\doiurl{10.1016/S0747-7171(88)80011-7}
\end{barticle}
\endbibitem

\bibitem[\protect\citeauthoryear{McCallum}{1988}]{bib_bib9}
\begin{barticle}
\bauthor{\bsnm{McCallum}, \binits{S.}}:
\batitle{An improved projection operation for cylindrical algebraic decomposition of three-dimensional space}.
\bjtitle{Journal of Symbolic Computation}
\bvolume{5}(\bissue{1--2}),
\bfpage{141}--\blpage{161}
(\byear{1988})
\doiurl{10.1016/S0747-7171(88)80010-5}
\end{barticle}
\endbibitem

\bibitem[\protect\citeauthoryear{McCallum}{1999}]{bib_bib10}
\begin{botherref}
\oauthor{\bsnm{McCallum}, \binits{S.}}:
On projection in CAD-based quantifier elimination with equational constraints.
Paper presented at the International Symposium on Symbolic and Algebraic Computation, University of British Columbia, Vancouver, 29–31 July 1999
(1999)
\end{botherref}
\endbibitem

\bibitem[\protect\citeauthoryear{McCallum}{2001}]{bib_bib11}
\begin{botherref}
\oauthor{\bsnm{McCallum}, \binits{S.}}:
On propagation of equational constraints in CAD-based quantifier elimination.
Paper presented at the International Symposium on Symbolic and Algebraic Computation, Petersburg Department of Steklov Mathematical Institute, Saint Petersburg, 19--23 July 2001
(2001)
\end{botherref}
\endbibitem

\bibitem[\protect\citeauthoryear{Grigoriev}{1988}]{bib_bib12}
\begin{barticle}
\bauthor{\bsnm{Grigoriev}, \binits{Y.D.}}:
\batitle{Complexity of deciding tarski algebra}.
\bjtitle{Journal of Symbolic Computation}
\bvolume{5}(\bissue{1--2}),
\bfpage{65}--\blpage{108}
(\byear{1988})
\doiurl{10.1016/S0747-7171(88)80006-3}
\end{barticle}
\endbibitem

\bibitem[\protect\citeauthoryear{Strzebo\'{n}ski}{2000}]{bib_bib13}
\begin{barticle}
\bauthor{\bsnm{Strzebo\'{n}ski}, \binits{A.W.}}:
\batitle{Solving algebraic inequalities}.
\bjtitle{Mathematical Journal}
\bvolume{7}(\bissue{4}),
\bfpage{525}--\blpage{541}
(\byear{2000})
\end{barticle}
\endbibitem

\bibitem[\protect\citeauthoryear{Hong}{1990}]{bib_bib14_PhDthesis}
\begin{botherref}
\oauthor{\bsnm{Hong}, \binits{H.}}:
Improvements in cad-based quantifier elimination.
PhD thesis,
the Ohio State University
(1990)
\end{botherref}
\endbibitem

\bibitem[\protect\citeauthoryear{Renegar}{1992a}]{bib_bib16}
\begin{barticle}
\bauthor{\bsnm{Renegar}, \binits{J.}}:
\batitle{On the computational complexity and geometry of the first-order theory of the reals. {P}art {I}: {I}ntroduction. {P}reliminaries. {T}he geometry of semi-algebraic sets. {T}he decision problem for the existential theory of the reals}.
\bjtitle{Journal of Symbolic Computation}
\bvolume{13}(\bissue{3}),
\bfpage{255}--\blpage{299}
(\byear{1992})
\doiurl{10.1016/S0747-7171(10)80003-3}
\end{barticle}
\endbibitem

\bibitem[\protect\citeauthoryear{Renegar}{1992b}]{bib_bib17}
\begin{barticle}
\bauthor{\bsnm{Renegar}, \binits{J.}}:
\batitle{On the computational complexity and geometry of the first-order theory of the reals. {P}art {II}: {T}he general decision problem. {P}reliminaries for quantifier elimination}.
\bjtitle{Journal of Symbolic Computation}
\bvolume{13}(\bissue{3}),
\bfpage{301}--\blpage{327}
(\byear{1992})
\doiurl{10.1016/S0747-7171(10)80004-5}
\end{barticle}
\endbibitem

\bibitem[\protect\citeauthoryear{Renegar}{1992c}]{bib_bib18}
\begin{barticle}
\bauthor{\bsnm{Renegar}, \binits{J.}}:
\batitle{On the computational complexity and geometry of the first-order theory of the reals. {P}art {III}: Quantifier elimination}.
\bjtitle{Journal of Symbolic Computation}
\bvolume{13}(\bissue{3}),
\bfpage{329}--\blpage{352}
(\byear{1992})
\doiurl{10.1016/S0747-7171(10)80005-7}
\end{barticle}
\endbibitem

\bibitem[\protect\citeauthoryear{Basu et~al.}{1996}]{bib_bib19}
\begin{barticle}
\bauthor{\bsnm{Basu}, \binits{S.}},
\bauthor{\bsnm{Pollack}, \binits{R.}},
\bauthor{\bsnm{Roy}, \binits{M.-F.}}:
\batitle{On the combinatorial and algebraic complexity of quantifier elimination}.
\bjtitle{Journal of the ACM}
\bvolume{43}(\bissue{6}),
\bfpage{1002}--\blpage{1045}
(\byear{1996})
\doiurl{10.1145/235809.235813}
\end{barticle}
\endbibitem

\bibitem[\protect\citeauthoryear{Basu et~al.}{1999}]{bib_bib20}
\begin{barticle}
\bauthor{\bsnm{Basu}, \binits{S.}},
\bauthor{\bsnm{Pollack}, \binits{R.}},
\bauthor{\bsnm{Roy}, \binits{M.-F.}}:
\batitle{Computing roadmaps of semi-algebraic sets on a variety}.
\bjtitle{Journal of the American Mathematical Society}
\bvolume{13}(\bissue{1}),
\bfpage{55}--\blpage{82}
(\byear{1999})
\doiurl{10.1090/S0894-0347-99-00311-2}
\end{barticle}
\endbibitem

\bibitem[\protect\citeauthoryear{Brown}{2001}]{bib_bib23}
\begin{barticle}
\bauthor{\bsnm{Brown}, \binits{C.W.}}:
\batitle{Simple cad construction and its applications}.
\bjtitle{Journal of Symbolic Computation}
\bvolume{31}(\bissue{5}),
\bfpage{521}--\blpage{547}
(\byear{2001})
\doiurl{10.1006/jsco.2000.0394}
\end{barticle}
\endbibitem

\bibitem[\protect\citeauthoryear{Dolzmann et~al.}{2004}]{bib_bib24}
\begin{botherref}
\oauthor{\bsnm{Dolzmann}, \binits{A.}},
\oauthor{\bsnm{Seidl}, \binits{A.}},
\oauthor{\bsnm{Sturm}, \binits{T.}}:
Efficient projection orders for CAD.
Paper presented at the International Symposium on Symbolic and Algebraic Computation, University of Cantabria, Santander, 4--7 July 2004
(2004)
\end{botherref}
\endbibitem

\bibitem[\protect\citeauthoryear{Brown and Mccallum}{2005}]{bib_bib25}
\begin{botherref}
\oauthor{\bsnm{Brown}, \binits{C.W.}},
\oauthor{\bsnm{Mccallum}, \binits{S.}}:
On using bi-equational constraints in CAD construction.
Paper presented at the 2005 International Symposium on Symbolic and Algebraic Computation, the Key Laboratory of Mathematics Mechanization, Chinese Academy of Sciences, Beijing, 24--27 July 2005
(2005)
\end{botherref}
\endbibitem

\bibitem[\protect\citeauthoryear{Strzebo\'{n}ski}{2006}]{bib_bib26}
\begin{barticle}
\bauthor{\bsnm{Strzebo\'{n}ski}, \binits{A.W.}}:
\batitle{Cylindrical algebraic decomposition using validated numerics}.
\bjtitle{Journal of Symbolic Computation}
\bvolume{41}(\bissue{9}),
\bfpage{1021}--\blpage{1038}
(\byear{2006})
\doiurl{10.1016/j.jsc.2006.06.004}
\end{barticle}
\endbibitem

\bibitem[\protect\citeauthoryear{Strzebo\'{n}ski}{2011}]{bib_bib27}
\begin{barticle}
\bauthor{\bsnm{Strzebo\'{n}ski}, \binits{A.W.}}:
\batitle{Cylindrical decomposition for systems transcendental in the first variable}.
\bjtitle{Journal of Symbolic Computation}
\bvolume{46}(\bissue{11}),
\bfpage{1284}--\blpage{1290}
(\byear{2011})
\doiurl{10.1016/j.jsc.2011.08.009}
\end{barticle}
\endbibitem

\bibitem[\protect\citeauthoryear{Brown}{2012}]{bib_bib28}
\begin{barticle}
\bauthor{\bsnm{Brown}, \binits{C.W.}}:
\batitle{Fast simplifications for tarski formulas based on monomial inequalities}.
\bjtitle{Journal of Symbolic Computation}
\bvolume{47}(\bissue{7}),
\bfpage{859}--\blpage{882}
(\byear{2012})
\doiurl{10.1016/j.jsc.2011.12.012}
\end{barticle}
\endbibitem

\bibitem[\protect\citeauthoryear{Hong and Din}{2012}]{bib_bib29}
\begin{barticle}
\bauthor{\bsnm{Hong}, \binits{H.}},
\bauthor{\bsnm{Din}, \binits{M.S.E.}}:
\batitle{Variant quantifier elimination}.
\bjtitle{Journal of Symbolic Computation}
\bvolume{47}(\bissue{7}),
\bfpage{883}--\blpage{901}
(\byear{2012})
\doiurl{10.1016/j.jsc.2011.05.014}
\end{barticle}
\endbibitem

\bibitem[\protect\citeauthoryear{Bradford et~al.}{2013}]{bib_bib30}
\begin{botherref}
\oauthor{\bsnm{Bradford}, \binits{R.}},
\oauthor{\bsnm{Davenport}, \binits{J.H.}},
\oauthor{\bsnm{England}, \binits{M.}},
\oauthor{\bsnm{McCallum}, \binits{S.}},
\oauthor{\bsnm{Wilson}, \binits{D.}}:
Cylindrical algebraic decompositions for boolean combinations.
Paper presented at the International Symposium on Symbolic and Algebraic Computation, Northeastern University, Boston, 26--29 June 2013
(2013)
\end{botherref}
\endbibitem

\bibitem[\protect\citeauthoryear{Basu et~al.}{2006}]{bib_bib21}
\begin{bbook}
\bauthor{\bsnm{Basu}, \binits{S.}},
\bauthor{\bsnm{Pollack}, \binits{R.}},
\bauthor{\bsnm{Roy}, \binits{M.-F.}}:
\bbtitle{Algorithms in Real Algebraic Geometry}.
\bpublisher{Springer},
\blocation{Berlin Heidelberg}
(\byear{2006})
\end{bbook}
\endbibitem

\bibitem[\protect\citeauthoryear{Brown}{2001}]{bib_bib22}
\begin{barticle}
\bauthor{\bsnm{Brown}, \binits{C.W.}}:
\batitle{Improved projection for cylindrical algebraic decomposition}.
\bjtitle{Journal of Symbolic Computation}
\bvolume{32}(\bissue{5}),
\bfpage{447}--\blpage{465}
(\byear{2001})
\doiurl{10.1006/jsco.2001.0463}
\end{barticle}
\endbibitem

\bibitem[\protect\citeauthoryear{Collins and Hong}{1991}]{bib_bib15}
\begin{barticle}
\bauthor{\bsnm{Collins}, \binits{G.E.}},
\bauthor{\bsnm{Hong}, \binits{H.}}:
\batitle{Partial cylindrical algebraic decomposition for quantifier elimination}.
\bjtitle{Journal of Symbolic Computation}
\bvolume{12}(\bissue{3}),
\bfpage{299}--\blpage{328}
(\byear{1991})
\doiurl{10.1016/S0747-7171(08)80152-6}
\end{barticle}
\endbibitem

\bibitem[\protect\citeauthoryear{Dolzmann and Sturm}{1997}]{bib129}
\begin{barticle}
\bauthor{\bsnm{Dolzmann}, \binits{A.}},
\bauthor{\bsnm{Sturm}, \binits{T.}}:
\batitle{Redlog: computer algebra meets computer logic}.
\bjtitle{Acm Sigsam Bulletin}
\bvolume{31}(\bissue{2}),
\bfpage{2}--\blpage{9}
(\byear{1997})
\doiurl{10.1145/261320.261324}
\end{barticle}
\endbibitem

\bibitem[\protect\citeauthoryear{Hong}{1992}]{bib127}
\begin{botherref}
\oauthor{\bsnm{Hong}, \binits{H.}}:
Simple solution formula construction in cylindrical algebraic decomposition based quantifier elimination.
Paper presented at the International Symposium on Symbolic and Algebraic Computation, University of California, Berkeley, 27--29 July 1992
(1992)
\end{botherref}
\endbibitem

\bibitem[\protect\citeauthoryear{Strzebo\'{n}ski}{2005}]{bib131}
\begin{botherref}
\oauthor{\bsnm{Strzebo\'{n}ski}, \binits{A.W.}}:
Applications of algorithms for solving equations and inequalities in Mathematica.
Paper presented at the Algorithmic Algebra and Logic, Universität Passau, Passau, 3--6 April 2005
(2005)
\end{botherref}
\endbibitem

\bibitem[\protect\citeauthoryear{Anai and Yanami}{2003}]{bib132}
\begin{botherref}
\oauthor{\bsnm{Anai}, \binits{H.}},
\oauthor{\bsnm{Yanami}, \binits{H.}}:
SyNRAC: A Maple-Package for Solving Real Algebraic Constraints.
Paper presented at SICE 2003 Annual Conference, Fukui University, Fukui, 4--6 August 2003
(2003)
\end{botherref}
\endbibitem

\bibitem[\protect\citeauthoryear{Davenport and Heintz}{1988}]{bib133}
\begin{barticle}
\bauthor{\bsnm{Davenport}, \binits{J.H.}},
\bauthor{\bsnm{Heintz}, \binits{J.}}:
\batitle{Real quantifier elimination is doubly exponential}.
\bjtitle{Journal of Symbolic Computation}
\bvolume{5}(\bissue{1--2}),
\bfpage{29}--\blpage{35}
(\byear{1988})
\doiurl{10.1016/S0747-7171(88)80004-X}
\end{barticle}
\endbibitem

\bibitem[\protect\citeauthoryear{Yang et~al.}{2001}]{bib134}
\begin{barticle}
\bauthor{\bsnm{Yang}, \binits{L.}},
\bauthor{\bsnm{Hou}, \binits{X.}},
\bauthor{\bsnm{Xia}, \binits{B.}}:
\batitle{A complete algorithm for automated discovering of a class of inequality-type theorems}.
\bjtitle{Science in China Series F: Information Sciences}
\bvolume{44},
\bfpage{33}--\blpage{49}
(\byear{2001})
\doiurl{10.1007/BF02713938}
\end{barticle}
\endbibitem

\bibitem[\protect\citeauthoryear{Xia and Yang}{2016}]{bib135}
\begin{bbook}
\bauthor{\bsnm{Xia}, \binits{B.}},
\bauthor{\bsnm{Yang}, \binits{L.}}:
\bbtitle{Automated Inequality Proving and Discovering}.
\bpublisher{World Scientific},
\blocation{Singapore}
(\byear{2016})
\end{bbook}
\endbibitem

\bibitem[\protect\citeauthoryear{Wen-Tsun}{1986}]{bib_Wu1958}
\begin{barticle}
\bauthor{\bsnm{Wen-Tsun}, \binits{W.}}:
\batitle{Basic principles of mechanical theorem-proving in elementary geometries}.
\bjtitle{Journal of Automated Reasoning}
\bvolume{2},
\bfpage{221}--\blpage{252}
(\byear{1986})
\doiurl{10.1007/BF02328447}
\end{barticle}
\endbibitem

\bibitem[\protect\citeauthoryear{Buchberger}{2006}]{bib_Gronber}
\begin{barticle}
\bauthor{\bsnm{Buchberger}, \binits{B.}}:
\batitle{Bruno buchberger’s phd thesis 1965: An algorithm for finding the basis elements of the residue class ring of a zero dimensional polynomial ideal}.
\bjtitle{Journal of Symbolic Computation}
\bvolume{41}(\bissue{3--4}),
\bfpage{475}--\blpage{511}
(\byear{2006})
\doiurl{10.1016/j.jsc.2005.09.007}
\end{barticle}
\endbibitem

\bibitem[\protect\citeauthoryear{Lazard and Rouillier}{2007}]{bib63}
\begin{barticle}
\bauthor{\bsnm{Lazard}, \binits{D.}},
\bauthor{\bsnm{Rouillier}, \binits{F.}}:
\batitle{Solving parametric polynomial systems}.
\bjtitle{Journal of Symbolic Computation}
\bvolume{42}(\bissue{6}),
\bfpage{636}--\blpage{667}
(\byear{2007})
\doiurl{10.1016/j.jsc.2007.01.007}
\end{barticle}
\endbibitem

\bibitem[\protect\citeauthoryear{Hong et~al.}{2015}]{bib_bib}
\begin{barticle}
\bauthor{\bsnm{Hong}, \binits{H.}},
\bauthor{\bsnm{Tang}, \binits{X.}},
\bauthor{\bsnm{Xia}, \binits{B.}}:
\batitle{Special algorithm for stability analysis of multistable biological regulatory systems}.
\bjtitle{Journal of Symbolic Computation}
\bvolume{70},
\bfpage{112}--\blpage{135}
(\byear{2015})
\doiurl{10.1016/j.jsc.2014.09.039}
\end{barticle}
\endbibitem

\bibitem[\protect\citeauthoryear{Chen et~al.}{2013}]{bib_49}
\begin{barticle}
\bauthor{\bsnm{Chen}, \binits{C.}},
\bauthor{\bsnm{Davenport}, \binits{J.H.}},
\bauthor{\bsnm{Maza}, \binits{M.M.}},
\bauthor{\bsnm{Xia}, \binits{B.}},
\bauthor{\bsnm{Xiao}, \binits{R.}}:
\batitle{Computing with semi-algebraic sets: Relaxation techniques and effective boundaries}.
\bjtitle{Journal of Symbolic Computation}
\bvolume{52},
\bfpage{72}--\blpage{96}
(\byear{2013})
\doiurl{10.1016/j.jsc.2012.05.013}
\end{barticle}
\endbibitem

\bibitem[\protect\citeauthoryear{R{\"o}st and Sadeghimanesh}{2021}]{bib14}
\begin{botherref}
\oauthor{\bsnm{R{\"o}st}, \binits{G.}},
\oauthor{\bsnm{Sadeghimanesh}, \binits{A.}}:
Unidirectional migration of populations with Allee effect.
https://api.semanticscholar.org/CorpusID:235655819
(2021)
\end{botherref}
\endbibitem

\bibitem[\protect\citeauthoryear{Arnon et~al.}{1984}]{bib_CAD}
\begin{barticle}
\bauthor{\bsnm{Arnon}, \binits{D.S.}},
\bauthor{\bsnm{Collins}, \binits{G.E.}},
\bauthor{\bsnm{McCallum}, \binits{S.}}:
\batitle{Cylindrical algebraic decomposition {I}: The basic algorithm}.
\bjtitle{SIAM Journal on Computing}
\bvolume{13}(\bissue{4}),
\bfpage{865}--\blpage{877}
(\byear{1984})
\doiurl{10.1137/0213054}
\end{barticle}
\endbibitem

\bibitem[\protect\citeauthoryear{Xia and Yang}{2002}]{bib3}
\begin{barticle}
\bauthor{\bsnm{Xia}, \binits{B.}},
\bauthor{\bsnm{Yang}, \binits{L.}}:
\batitle{An algorithm for isolating the real solutions of semi-algebraic systems}.
\bjtitle{Journal of Symbolic Computation}
\bvolume{34}(\bissue{5}),
\bfpage{461}--\blpage{477}
(\byear{2002})
\doiurl{10.1006/jsco.2002.0572}
\end{barticle}
\endbibitem

\end{thebibliography}
\end{document}